\RequirePackage[reqno]{amsmath}

\documentclass[11pt,reqno]{amsart}

\usepackage{graphicx,graphics,caption, subcaption, enumerate,color}

\usepackage{amsmath,amssymb,amsfonts,bm,empheq}
\usepackage{cases}

\usepackage{comment,soul,xcolor,todonotes}
\usepackage[normalem]{ulem}

\usepackage{lineno}

\usepackage{hyperref}
\hypersetup{
    unicode=false,          
    pdftoolbar=true,        
    pdfmenubar=true,        
    pdffitwindow=false,     
    pdfstartview={FitH},    
    pdftitle={My title},    
    pdfauthor={Author},     
    pdfsubject={Subject},   
    pdfcreator={Creator},   
    pdfproducer={Producer}, 
    pdfkeywords={keyword1} {key2} {key3}, 
    pdfnewwindow=true,      
    colorlinks=true,       
    linkcolor=red,          
    citecolor=green,        
    filecolor=magenta,      
    urlcolor=cyan           
}

\newtheorem{theorem}{Theorem}
\newtheorem{corollary}[theorem]{Corollary}
\newtheorem{definition}[theorem]{Definition}
\newtheorem{proposition}[theorem]{Proposition}

\newtheorem{lemma}[theorem]{Lemma}
\newtheorem{remark}{Remark}

\newcommand{\norm}[1]{\left\Vert#1\right\Vert}
\newcommand{\inpd}[2]{\left\langle #1, #2 \right\rangle}

\newcommand{\abs}[1]{\left\vert#1\right\vert}
\newcommand{\set}[1]{\left\{#1\right\}}

\newcommand{\wt}[1]{\widetilde{#1}}
\newcommand{\tM}{{\wt{M}}}
\newcommand{\tJ}{\wt{J}}
\newcommand{\tA}{\wt{A}}

\newcommand{\Real}{\mathbb {R}}

\newcommand{\tr}{\textsf{T}}

\newcommand{\eps}{\varepsilon}

\renewcommand{\d}{\ensuremath{\mathrm{d}}}
\newcommand{\dt}{ \ensuremath{\mathrm{d} t } }
\newcommand{\ds}{ \ensuremath{\mathrm{d} s } }

\newcommand{\dx}{ \ensuremath{\mathrm{d} x} }

\newcommand{\tomega}{\wt{\Omega}}

\newcommand{\Hmt}{\mathbb{H}}
\newcommand{\T}{\mathcal{T}}

\newcommand{\A}{\mathcal{A}}

\def\af {action functional }

\newcommand{\FW} {Freidlin-Wentzell }
\newcommand{\ldp} {large deviation principle}
\newcommand{\qp}{quasi-potential }

\usepackage{lineno}

\newcommand{\yu}[1]{{\color{purple}{#1}}}

\begin{document}
\begin{center}
{\bf \Large Quasi-potential Calculation and
Minimum Action Method for  Limit Cycle }

~
\par
~
\par
Ling Lin\footnote{email: linling27@mail.sysu.edu.cn.}
\par
School of Mathematics
\par
Sun Yat-sen University
\par
Guangzhou,   510275,  China

\par
\
\par
 Haijun Yu \footnote{email: hyu@lsec.cc.ac.cn.
 The research of H. Yu was supported by
 NNSFC Grant 11771439, 91530322 and Science Challenge Project No.~TZ2018001.}
 \par
School of Mathematical Sciences,
University of Chinese Academy of Sciences\\
NCMIS \& LSEC, Institute of
Computational Mathematics and Scientific/Engineering
Computing, Academy of Mathematics and Systems
Science, Beijing 100190, China

\par
\
\par

Xiang Zhou\footnote{email: xiang.zhou@city.edu.hk.
The research of XZ  was supported by the grants from the Research Grants Council of the Hong Kong Special Administrative Region, China 
(Project No. CityU 109113 and 11304715).  }%
\par
Department of Mathematics
\par
City University of Hong Kong\par
Tat Chee Ave, Kowloon \par Hong Kong SAR

\end{center}

 
\section*{Abstract}
  We study  the noise-induced escape from a stable limit
  cycle of a  non-gradient dynamical system driven by a
  small additive noise.
  The  fact that the optimal transition path in this case is infinitely long
  imposes a severe numerical challenge to resolve it in  the minimum action method.
We first consider the landscape of the
  quasi-potential near the limit cycle, which  characterizes
   the minimal cost of the noise to drive the
  system far away form the limit cycle.    We derive and compute the quadratic
  approximation of this \qp near the limit cycle in the form
  of a positive definite solution to a matrix-valued
  periodic Riccati differential equation on the limit cycle.
  We then combine this local approximation in the neighbourhood of the limit cycle with the
   minimum action method applied outside of the neighbourhood.
   The neighbourhood size is selected to be compatible with the   path discretization error.  
   By several numerical examples,
  we show that this  strategy  effectively  improve the minimum action
  method to compute the {spiral} optimal
  escape path from {limit cycles} in
{various systems}.

{{\bf Keywords}: rare event,  non-gradient system, quasi-potential, limit cycle,  minimum action method
}
\

{{\bf  Mathematics Subject Classification (2010)}  Primary 65K05, Secondary  82B05	  }

\section{Introduction}

Many physical and biological systems exhibit sustainable
oscillating dynamics and are altered by random external
perturbations simultaneously.  To understand the long term
impact of noise on the stable oscillations in a
deterministic dynamics is an important question.  We here
consider a continuous-time dynamical system exhibiting a
stable limit cycle, subject to the additive random
perturbation in the small noise limit.  The model is the
following Ito stochastic differential equation in $\Real^d$:
\begin{equation}\label{SDE}
\dx_t = b(x_t)\dt + \sqrt{\eps}\sigma(x_t)\d w_t, ~~\quad \eps \ll1,
\end{equation}
where $b: \Real^d \to \Real^d$ is a smooth drift vector
field, $w_t$ is the standard $\Real^{d}$-valued Wiener
  process, $\sigma: \Real^d \to \Real^d$ is a  matrix-valued function
  and the positive semidefinite matrix
  $a(x):=\sigma(x) \sigma (x)^\tr$ is usually known as the
  diffusion tensor.  We are concerned with the case that
the deterministic dynamical system $\dot{x}=b(x)$ has a
stable limit cycle $\Gamma$.  This model of noisy perturbed
stable oscillations has attracted many interests in the
areas of nonlinear oscillators in biology, synchronization
of neural network dynamics, fluid dynamics and so on\cite{KuramotoBook,McCbook1989,Bressloff2014stocell,WanYE_NS2dMAM_2015,WanY_NS2dtMAM_2017,
PRL2018MAPOsc}.
The central question in concern  is how the stochastic
trajectory of \eqref{SDE} is driven far away from $\Gamma$
due to the long time effect of the noise.  Such
non-equilibrium behaviors correspond to many important rare
events and the stability problems of the stochastic systems.

It is a classic problem on the exit   from a domain
in a non-equilibrium system and there are many analytical and experimental  studies in physics literature on this topic.  
The  asymptotic analysis works   \cite{MatSchuss1982, Day1993,Maier1996PRL,Berglund2004}
have focused on the exit problem where the basin boundary is a closed curve or an unstable limit cycle.
  For the case
considered here on the noise-induced escape from {  stable} limit cycles,
 \cite{Holland1978} and \cite{Kurrer1991} studied the 
the invariant measure near the
stable limit cycle in small noise intensity limit.
With the numerical experiments,    \cite{McC2005} is concerned with
the optimal trajectories in stochastic continuous dynamical systems
and maps.  The approach \cite{McC2005} is to study the activation energy
by solving the underlying Hamiltonian system, which is equivalent in mathematics to
the quasi-potential in our approach here based on the \FW
\ldp.

The large deviation theory provides a useful tool for  the 
noise-induced problems for \eqref{SDE}  in the asymptotic regime of small noise.
  The mathematical theory of \FW
\ldp\cite{FW2012} states that the most probable trajectory
of \eqref{SDE} between two given points $a_1$ and $a_2$ is
the minimizer of the \FW action functional:  the minimizer is called the minimum action path
(MAP) \cite{weinan-MAM2004} and the minimum value of the
action is the so called {\it quasi-potential}.  This theory is also applicable to the
transitions from a  {compact invariant} set $K_1$ to
another  set $K_2$.  The
quasi-potential landscape, denoted as $V(x)$, takes the zero
value at $K_1$ and increases its value away from $K_1$,
intuitively depicting the cost that the noise has to pay to
drive the system to reach the target point. 
So, the quasi-potential is
an important  {quantity}  describing  the landscape
 of the minimal  action for the escape from    $K_1$. 
In the case that $K_1$
is a stable limit cycle, the level set of the quasi-potential  near this limit cycle  is very useful to understand the
effect of noise in driving the periodic system \eqref{SDE}
in the long run time.

It is shown that $V$ satisfies a Hamilton-Jacobi equation \cite{FW2012}.  For planar problems $d=2$,
one can numerically solve this Hamilton-Jacobi equation with
suitable numerical schemes \cite{Cameron:2012physicaD}.  In
general, this problem has to be  solved by the variational approach
(the least action principle) rather than the PDE approach.
The minimum action method (MAM) and its variants
\cite{weinan-MAM2004,aMAM2008,Heymann2006,Heyman2008,Wan2011}
have been developed to directly calculate the minimum action path
(MAP).  In practice, the MAM works on a path with two
 {fixed} endpoints, so it naturally fits the situation
where both $K_1$ and $K_2$ are singletons.  But when $K_1$
is a continuum set, for instance, a limit cycle $\Gamma$,
then the challenge for the transition from $K_1$ to $K_2$ is
that every point in $K_1$ is equally important since the \qp
is zero everywhere inside $K_1$, which usually implies that 
the actual MAP may have an infinite length. 
This key fact can   also be observed from the  underlying Hamiltonian flow. 
The extremal path, together with its  corresponding  momentum
part, satisfies a Hamiltonian flow.
When the momentum term in the Hamiltonian flow vanishes,
one has a flow identical to the original dynamics $\dot{x}=b(x)$.
So the stable set $K_1$ in the original dynamics becomes the $\alpha$-limit set of the extremal {\it escape} path with non-vanishing momentum.
 This   suggests that the
optimal exit path emitting from the limit cycle has an
infinite arc length as $t\to -\infty$.

The numerical challenge in practical computation is that 
it is not possible to resolve an
infinitely long path perfectly.  The previous study
on the Kuramoto-Sivashinsky equation in 
  \cite{KS-WZE2009} is to select an arbitrary point on
the travelling wave and use the transition from this point
to approximate the transition from the travelling wave.
This approach works reasonably well  if the main purpose is
to explore the high dimensional phase space rather than a
pursuit of the precise values of   quasi-potential, since the
accuracy of the path deteriorates  critically  only near the
limit cycle.  In addition, the minimum action obtained in
this way depends crucially on the initial guess in the
optimization: the more loops around the limit cycle in the
initial path, the better accuracy of the numerical results,
but the length of the numerical path becomes longer and
longer.

In this paper, we propose a new computational strategy  
of the MAM
to adaptively compute the MAP from the limit cycle.
Our new method is based on the explicit form of the 
quadratic approximation of 
  the quasi-potential near the limit
cycle. To this end, we construct a small tube around the
limit cycle, selected by a given numerical tolerance compatible with the 
discretization of the path. The MAM is only applied to the
outside of this tube and the true path spiralling outward
with infinite length is truncated to have a finite length
with a new initial point confined  {on the surface
  of} the tube.  To construct the analytic form of the \qp
in the tubular neighborhood of the limit cycle, the \qp is
approximated by a quadratic form with a $(d-1)\times(d-1)$
positive definite matrix $G$ along the limit cycle.  $G$ is
computed  by solving a periodic Riccati differential equation
(PRDE), which is not a challenging numerical problem even
for a large dimension $d$.  The existence-and-uniqueness condition of
the positive definite solution to the PRDE is shown to be
closely connected to the linear stability of the limit cycle
and the non-degeneracy of the diffusion tensor.  If $d=2$,
we have the analytic solution of $G$ explicitly.  The
eigenvalues of $G$ along the limit cycle describe the
varying widths of the tubular level set of the
quasi-potential; the eigenvectors of $G$ lying in the normal
plane of the limit cycle tell us which direction is more
preferred (or less preferred) for the stochastic trajectory
to depart from the limit cycle.  The asymptotic
approximation of the quasi-potential can also provide
the correct initial values if one want to solve
the Hamiltonian system to calculate the quasi-potential
along the Hamiltonian trajectory.

In the following, Section 2 will review the basics of
several theoretic foundations.  Section 3 derives the
approximation form of the \qp near the limit cycle. Section
4 is devoted to the Riccati matrix differential equations.
Our  main numerical method is presented in Section 5,
followed by several numerical examples in Section 6.  The
last section is our conclusive part.

\section{Review}
\subsection{Quasi-potential
	and minimum action method}

Assume $\gamma(\tau)$ is a   periodic solution of the deterministic dynamics
\begin{equation}\label{eqn:ODE}
\dot{x}(\tau)=b(x)
\end{equation}
with a least period $\T>0$.  Then the trajectory
$\Gamma :=\set{\gamma(\tau): \tau\in [0,~\T]}$ in the phase
space is a limit cycle.  $\Gamma$ is assumed to be stable in
the sense which will be specified later.
Then the SDE   \eqref{SDE} is a random perturbation of \eqref{eqn:ODE}.
We are interested in the \qp for the noise perturbed escape from  this stable limit cycle:
\begin{equation}\label{qp}
V(x):=\inf_{T>0}  \inf_{\phi(-T)\in \Gamma,\phi(T)=x} S_T[\phi],
\end{equation}
where the \FW \af $S_T$ associated with an interval
$[-T,~T]$ is defined by
\begin{equation}\label{action}
S_T[\phi]=\frac12 \int_{-T}^T \norm{\dot{\phi}-b(\phi)}^2_{a(\phi(\tau))} \d\tau,
\end{equation}
for an absolute continuous function $\phi$; otherwise,
$S_T[\phi]=+\infty$.  Here
\[
a:=\sigma\sigma^\tr, ~~\mbox{ and } ~~ \norm{v}_a :=\sqrt{ v^\tr {a}^{-1} v},
~~
\inpd{u}{v}_a:=\inpd{u}{a^{-1}v}=u^\tr a^{-1} v.
\]
\begin{remark}
	If $a$ is not invertible, then  the above \af is modified as follows
	\begin{equation}\label{action2}
	S_T[\phi]=\frac12 \inf_{\sigma u =\dot{\phi}-b(\phi) }\int_{-T}^T  \norm{u}^2  \d\tau.
	\end{equation}
\end{remark}

The \qp $V(x)\geq 0$ and the equality holds on the limit cycle $\Gamma$.
It has  been shown  \cite{FW2012} that $V$
satisfies the Hamilton-Jacobi equation in the basin of attraction of $\Gamma$:
\begin{equation}\label{eqn:H}
\Hmt(x,\nabla V(x))=0,
\end{equation}
where the Hamiltonian
\begin{equation}
\label{def:H}
\Hmt(x,p):=\inpd{b(x)}{p}+\frac12\inpd{p}{a(x)p}.
\end{equation}
The extremal path of the
variational problem \eqref{qp} satisfies the canonical equations
of the
Hamiltonian   system:
\begin{equation} \label{eqn:HODE}
\begin{cases}
\dot{\phi}= \Hmt_p(\phi,p)=b(\phi)+a(\phi)p,\\
\dot{p}=-\Hmt_x(\phi,p) = - (\partial_x b(\phi))^\tr p-\frac12\partial_x\inpd{p}{ a(\phi)p}.
\end{cases}
\end{equation}
where $[\partial_x b(x)]_{ij} =
\left [\partial_{ x_j } b_i (x) \right]$ is the Jacobian matrix of the vector field $b$.
Then the quasi-potential
along the extremal path
$\phi$ can be calculated by
\[
V(\phi(t)) = \frac12 \int_0^t
\norm{a(\phi(t'))p(t')}^2_a\dt'
=\frac12 \int_0^t  \inpd{p(t')}{a(\phi(t')) p(t')}\dt'.
\]
If  caustic arises, then multiple extremal
paths may intersect at some points, and the true value of
the \qp $V$ at these focusing points is the minimum
of the multiple values arising from the multiple extremal paths.

For the transition path escaping the limit cycle $\Gamma$,
the initial condition of \eqref{eqn:HODE} should be imposed
at $t\to -\infty$ as follows:
$\underset{t\to-\infty}\lim \mbox{dist}(\phi(t),\Gamma) \to 0
$ and $  \underset{t\to-\infty}\lim p(t)=0.
$
``$\mbox{dist}$'' is the Hausdorff distance between two
sets.  But
$\Gamma \cap \set{\phi(t)}=\emptyset$ for any $t\in \Real$.
In fact, the extremal path winds around the limit cycle and
follows the same rotation direction as $\Gamma$.  This means
that the extremal path has an infinite length.

The geometric \af   \cite{Heymann2006} is based on
the Maupertuis's principle \cite{Landau-Lifshitz-Mechanics}
and written in terms of an arbitrarily parametrized
geometric curve:
\begin{equation}
\hat{S}[\varphi] = \int_{\varphi}  \inpd{p}{\d \varphi}
=\int_{\varphi}  \inpd{(\dot{\phi}-b(\phi))}{\d \varphi}_a
\end{equation}
where the curve $\varphi$ is a geometric description for the
time variable function $\phi(t)$.  The momentum
$p(t)=a(\phi(t))^{-1} \left( \dot{\phi}-b(\phi) \right)$ is
related to the \qp by $ p(t)=\nabla V( \phi(t)).$ Here
$\dot{\phi}(t)$ is the time derivative of $\phi$, which has
to be provided additionally in $\hat{S}$.  If the path
$\varphi$ has a finite length $L$, then it can be
parametrized by its arc-length $\varphi(s), -L\leq s\leq 0$.
The path then has two parametrized forms: $\phi(t)$ and
$\varphi(s)$.  The optimal change-of-variable between time
$t$ and arc-length $s$ is obtained by Maupertuis's
principle. Equivalently, the result corresponds to the
zero-Hamiltonian property, $\Hmt(\varphi,p)=0$, which is
further equivalent to the important
identity $$\|{\dot{\phi}\| }_a=\norm{b(\phi)}_a.$$ So,
$\ds /\dt =\frac{\norm{b}_a}{\norm{\varphi'}_a}$, where
$\varphi'(s)$ is the derivative of $\varphi$ parametrized by
the arc-length parameter $s$.  The geometric minimum action
method \cite{Heymann2006} (gMAM)  is based on this new
variational problem of minimizing $\hat{S}$.

The \qp defined in \eqref{qp} then is equivalent to
\begin{equation}\label{gqp}
V(x) = \inf_{ L>0}\inf_{{\varphi\in AC[-L,~0]}\atop{\varphi(-L)\in \Gamma, \varphi(0)=x}} \hat{S}[\varphi],
\end{equation}
where $AC[-L,~0]$ denotes the space of absolutely continuous
functions on $[-L, ~0]$.  As noted above, since the set
$\Gamma$ is a limit cycle, the arc length $L$ of the optimal
path is infinite.  In numerical computation based on
\eqref{qp} or \eqref{gqp}, the path has to be truncated by a
finite value of either time $T$ or arc length $L$.  Thus,
the gMAM inevitably produces truncation errors as any other
version of MAM .  The tMAM
\cite{WantMAM2015,WanY_NS2dtMAM_2017,WanYZ_tMAMconv_2017}
strives to adaptively match the truncation errors with the
numerical optimization errors by
 determining  a larger and larger $T$ as the
resolution of the path is finer and finer.  This method
works quite remarkably in practice and we think the same
idea can also be applied to the gMAM for the case of
infinite arc length $L$.  But the numerical stiffness
increases with the interval length $T$ (or $L$).  Our idea
here to avoid this stiffness due to infinite long interval
length is to find an approximation to the \qp within a tube
$\set{x: \mbox{dist}(x,\Gamma) \leq \delta}$ and the
numerical computation is only applied to the outside of this
tube.

\subsection{Quadratic approximation of the quasi-potential at stationary points.}
The asymptotic idea has been implemented to analyze the
Hamiltonian flow \eqref{eqn:HODE} around stationary points.
We briefly review this result \cite{BouchetPert2016} on the
approximation of the \qp at a stable stationary point.   
Assume that the linearized dynamics at a stationary point
$x_*$ (not necessarily stable) of $b$ is $\dot{x}= J x$
where $J$ is the Jacobian matrix $\partial_x b$ at $x_*$.
$J$ is assumed to be non-degenerated.  Assume that
$V(x)=\frac12 x^\tr A x$, where $A$ is a symmetric matrix to
be determined.  Then the zero-Hamiltonian condition
$\Hmt(x, \nabla V)=0$ and \eqref{eqn:HODE} lead to the
matrix equation $AJ+J^\tr A+ {A}{aA}=0$.  The unique
positive definite matrix solution $A$ satisfies
$A^{-1}= \int_0^\infty e^{tJ} a e^{tJ^\tr} \dt$ (where $a$
is valued at $x_*$).  The tangent flow of the Hamiltonian
system \eqref{eqn:HODE} in $\Real^{2d}$ is then in the
following two subspaces: $ (\dot{x},\dot{p}) = (Jx, 0) $ and
$(\dot{x},\dot{p} ) = (-(J+A)x, -J^\tr Ax) $.  If $J$
corresponds to a stable fixed point (all eigenvalues have
negative real parts), then $A$ is positive definite.

The   generalization of  this asymptotic  analysis
to the setting for a stable limit cycle
is more complicated than this fixed point case.
We first need  set up some local  curvilinear coordinates
around the limit cycle
by using a moving affine frame along the limit cycle.

\subsection{Linear stability of limit cycle}
\label{ssec:LSLL}
We review some classic concepts for asymptotic
stability of periodic ordinary differential equations in the
Floquet theory \cite{Coddington1955}. Some of them will be used later.
\begin{definition}\label{def:Phi}
  Let $M(\cdot)\in\Real^{n\times n}$ be a $\T$-periodic
  matrix-valued continuous function.  Fix an initial time
  $\tau_0$.  $\Phi_M(\tau,\tau_0)$ is called the state
  transition matrix associated with ${M}(\cdot)$ if it
  solves the periodic matrix differential equation
  \begin{equation} \label{eqn:Phi}
	\begin{cases}
      \dfrac{\partial }{\partial \tau}\Phi_{M}(\tau,\tau_0)={M}(\tau)\Phi_{M}(\tau,\tau_0),\\
      \Phi_{M}(\tau_0,\tau_0)=I,
	\end{cases}
  \end{equation}
  where $I$ is the identity matrix.  The monodromy matrix at
  time $\tau$ is defined as
  \begin{equation} \label{eqn:monodromy}
    \bar{\Phi}_M(\tau):=\Phi_M(\tau+\T,\tau).
  \end{equation}
  The eigenvalues of the monodromy matrix
  $\bar{\Phi}_M(\tau)$ are independent of $\tau$ because
  $\bar{\Phi}_M(\tau)=\Phi_M(\tau,0)\bar{\Phi}_M(0)\Phi_M^{-1}(\tau,0)$.
  The eigenvalues of the monodromy matrix
  $\bar{\Phi}_M(\tau)$ are called the characteristic
  multipliers for the linear $\T$-periodic system
  $\dot{y}(\tau)=M(\tau)y(\tau)$, or simply the
  characteristic multipliers of $M(\cdot)$.
\end{definition}

Recall that $\gamma(\tau)$ is a $\T$-periodic solution of
\eqref{eqn:ODE}, i.e.,
\begin{equation}\label{eqn:limit cycle}
  \dot{\gamma}(\tau)=b(\gamma(\tau)).
\end{equation}
Linearization of \eqref{eqn:ODE} around $\gamma(\tau)$ leads
to the following periodic linear system
\begin{equation}\label{eqn:first var}
  \dot{y}=\partial_x b(\gamma(\tau))y.
\end{equation}
By differentiating \eqref{eqn:limit cycle}, we see that
$\dot{\gamma}(\tau)$ is a $\T$-periodic solution of the
equation \eqref{eqn:first var}. It follows that the linear
$\T$-periodic system \eqref{eqn:first var} always has a
characteristic multiplier equal to $1$.  The stability of
the limit cycle is characterized by the other $(d-1)$
characteristic multipliers.

\begin{definition}\label{def:stability of limit cycle}
  We say that the limit cycle $\Gamma$, which is the orbit
  of a $\T$-periodic solution $\gamma(\tau)$ of
  \eqref{eqn:ODE}, is asymptotically stable if, except for
  the trivial characteristic multiplier $1$, the other $d-1$
  characteristic multipliers of \eqref{eqn:first var} lie
  inside the open unit disk in the complex plane.
\end{definition}

The asymptotic stability of the limit cycle defined in
Definition \ref{def:stability of limit cycle} has the
property of asymptotic orbital stability in the sense that
any solution of \eqref{eqn:ODE} which comes near the limit
cycle will tend to the limit cycle as $\tau\to+\infty$.
Refer to Theorem 2.2 of Chapter 13 in the textbook \cite{Coddington1955}.

\subsection{Curvilinear coordinates}
\label{ssec:gradc}
To set up the curvilinear coordinates around the limit cycle
$\Gamma$ in $\Real^d$, we need a moving affine frame along
the limit cycle $\Gamma$ in $\Real^d$, i.e., a collection of
$d$ differentiable $\T$-periodic mappings
$e_i:[0,~\T]\longrightarrow \Real^d$, $0\leq i\leq d-1$,
such that for all $\tau\in[0,~\T]$, the (column) vector set
$\{e_i(\tau):0\leq i\leq d-1\}$ forms a basis of $\Real^d$.
Each $e_i(\tau)$ may be viewed as a vector field along the
limit cycle $\Gamma$.  In addition, we set $e_0(\tau)$ to be
a tangent \yu{unit} vector field along the limit cycle
$\Gamma$, i.e.,
\begin{equation}\label{eqn:e0}
  e_0(\tau)=\lambda(\tau)\dot{\gamma}(\tau),
\end{equation}
where a superior dot denotes differentiation with respect to
$\tau$, and $\lambda(\tau)$ is a $\T$-periodic nonzero
scalar function.  To be well-defined, we need assume that
$\dot{\gamma}(\tau)$ never vanishes for any $\tau$, which
amounts to saying that the vector field $b(x)$ never
vanishes on the limit cycle $\Gamma$.  Given a moving affine
frame $e_i(\tau)$, $0\leq i\leq d-1$, the following
equations for the derivatives hold:
\begin{equation}\label{eqn:general Frenet--Serret}
  \dot{e}_{j}(\tau)=\sum_{i=0}^{d-1}\omega_j^i(\tau)e_{i}(\tau),  \quad \text{for } 0 \leq   j \leq  d-1,
\end{equation}
where
\begin{equation}\label{def:omega}
  \omega_j^i(\tau)=\inpd{e^i(\tau)}{\dot{e}_{j}(\tau)}, \quad \text{for } 0 \leq  i, j \leq  d-1,
\end{equation}
and the (column) vector set $\{e^i(\tau): 0\leq i\leq d-1\}$
is the reciprocal basis for the basis
$\{e_i(\tau): 0\leq i\leq d-1\}$, that is
\begin{equation}\label{eqn:ee}
  \inpd{e^i(\tau)}{ e_j(\tau)}=\delta^i_j:=\begin{cases}
    1,  \quad & \text{if } i=j,\\
    0, & \text{otherwise}.
  \end{cases}
\end{equation}
So, the normal plane of the limit cycle $\Gamma$ is
$P(\tau):=\mbox{span}\set{e^1(\tau),\ldots, e^{d-1}(\tau)} =
(e_0(\tau))^\perp = (\dot{\gamma}(\tau))^\perp$.
Actually this normal plane is the $d-1$ dimensional direct
sum of the generalized left eigenspaces for the nontrivial
eigenvalues (excluding 1) of the monodromy matrix
$\bar{\Phi}_{\partial_x b(\gamma(\cdot))}(\tau)$ associated
with the Jacobian $\partial_x b(\gamma(\tau))$.

Let
$E(\tau)=\begin{bmatrix}e_0(\tau), & \ldots ,&
  e_{d-1}(\tau)\end{bmatrix}$
denote the $d\times d$ matrix whose columns are the vectors
$e_i(\tau)$, $0\leq i\leq d-1$. Then its inverse matrix is
$ E(\tau)^{-1}=\begin{bmatrix} e^0(\tau) & \cdots &
  e^{d-1}(\tau)
\end{bmatrix}^\tr, $
and \eqref{eqn:general Frenet--Serret} can be written in the
matrix form
\begin{equation}\label{eqn:general FS matrix form}
  \dot{E}(\tau)=E(\tau)\Omega(\tau),
\end{equation}
or
\begin{equation}\label{eqn:Omega}
  \Omega(\tau)=E(\tau)^{-1} \dot{E}(\tau).
\end{equation}
where the element in the $(i+1)$-th row and $(j+1)$-th
column of $\Omega(\tau)$ is $\omega^i_j(\tau)$.

\begin{remark} \label{rem:b} If one assumes that the basis
  $\{e_i(\tau): 0\leq i\leq d-1\}$ is an orthonormal basis,
  then $e^i(\tau)=e_i(\tau)^\tr$ and $\abs{e_i}=1$,
  $0\leq i\leq d-1$. It follows that
  $\omega^i_j=\inpd{e_i}{ \dot{e}_j}$ satisfies
	\[
	\omega_i^j(\tau)=-\omega_j^i(\tau), \quad \text{for } 0
    \leq i, j \leq d-1, 0\leq \tau\leq \T,
	\]
	i.e., $\Omega(\tau)$ is antisymmetric.
	
	Assume $\{\gamma(\tau): 0\leq \tau\leq \T\}$ is a curve
    of order $d$, i.e., for all $\tau$, the $k$-th
    derivative $\gamma^{(k)}(\tau)$, $1\leq k\leq d$, are
    linearly independent, then we can construct the moving
    frame $e_i(\tau)$, $0\leq i\leq d-1$ from the
    derivatives of $\gamma(\tau)$ by using the Gram-Schmidt
    orthogonalization process. Consequently, the resulting
    frame $e_i(\tau)$, $0\leq i\leq d-1$ satisfies that for
    every $\tau$, the vector set
    $\{e_i(\tau):0\leq i\leq d-1\}$ forms an orthonormal
    basis in $\Real^d$, and in addition, for all
    $1\leq k\leq d$, the $k$-th derivative
    $\gamma^{(k)}(\tau)$ of $\gamma(\tau)$ lies in the span
    of the first $k$ vectors $e_i(\tau)$, $0\leq i\leq k-1$.
    It then follows that $\Omega(\tau)$ is antisymmetric and
    tridiagonal:
	\[
	\Omega(\tau)= \begin{bmatrix}
	0 &  \omega_1(\tau) & &0 \\
	-\omega_1(\tau) &  \ddots &\ddots & \\
	& \ddots&  0 & \omega_{d-1}(\tau) \\
	0 &  & -\omega_{d-1}(\tau)& 0\\
	\end{bmatrix}.
	\]
	The frame  $\set{e_i(\tau), 0\leq i\leq d-1}$
	constructed in this way  is called the Frenet frame,
	and the corresponding  equation \eqref{eqn:general Frenet--Serret} or
	\eqref{eqn:general FS matrix form}
	is known as the Frenet--Serret formula, and the invariant 
	$\kappa_i(\tau):=\omega_i(\tau)/\abs{\dot{\gamma}(\tau)}$ 
    is called the $i$-th curvature of the curve $\gamma(\tau)$ and can be
	determined by $\set{e_i(\tau):0\leq i\leq d-1}$.
	
\end{remark}

\subsection{Gradient form in terms of
	curvilinear coordinates}
  Now, equipped with an affine frame $e_i(\tau)$,
  $0\leq i\leq d-1$ defined on the limit cycle $\Gamma$ as
  stated in \S \ref{ssec:gradc} (without the requirement of
  orthonormality in Remark \ref{rem:b}), we introduce a set
  of local curvilinear coordinates
\begin{equation}\label{eqn:cc}
  (\tau,z)
  \in [0,~\T)\times \Real^{d-1},\quad
  \yu{z=(z^1,\cdots,z^{d-1})^\tr\in\Real^{d-1}}
\end{equation}
in the tubular neighborhood of the limit cycle $\Gamma$ by
writing a point $x$ in the neighborhood as
\begin{equation}\label{eqn:coord}
  x(\tau,z):=\gamma(\tau)+\sum_{j=1}^{d-1} z^{j} e_{j}(\tau)
  =\gamma(\tau)+ \wt{E}(\tau) z.
\end{equation}
Here the ${d\times(d-1)}$ matrix $\wt{E}$ is
$\wt{E}:=\begin{bmatrix}e_1(\tau), & \ldots ,&
  e_{d-1}(\tau)\end{bmatrix}.$
Differentiating \eqref{eqn:coord} with respect to $\tau$ and
$z_i$ yields
\begin{align*}
  \partial_\tau x&
     =\biggl[\lambda(\tau)^{-1}+\sum_{j=1}^{d-1} z^{j}\omega^{0}_j(\tau) \biggr]e_0(\tau)+\sum_{i,j=1}^{d-1} z^{j} \omega^{i}_j(\tau)e_i(\tau)
  \\
                 &=
                   \left[\lambda(\tau)^{-1}+ \inpd{e^0(\tau)}{\dot{\wt{E}}(\tau)z} \right]
                   e_0(\tau)+ \wt{E}(\tau)\wt{\Omega}(\tau)z,
  \\
  \partial_{z^i} x&=e_{i}(\tau),\quad \text{for } 1\leq i\leq d-1.
\end{align*}
$\tomega(\tau)$ denotes the $(d-1)\times (d-1)$ submatrix of
$\Omega(\tau)$ by deleting the first row and the first
column.
The Jacobian matrix
$\begin{bmatrix}\partial_\tau x, & \partial_{z^1} x,&\cdots
  & \partial_{z^{d-1}} x\end{bmatrix}$
at $(\tau,z=0)$ of the mapping defined by \eqref{eqn:coord}
is the matrix whose columns are the vectors
$\lambda(\tau)^{-1}e_0(\tau)$, $e_1(\tau)$, $\ldots$,
$e_{d-1}(\tau)$, therefore it is non-singular for any
$\tau$. So the mapping defined by \eqref{eqn:coord} is a
local diffeomorphism and gives a differentiable
transformation between the Cartesian coordinates $x$ and the
curvilinear coordinates $(\tau,z)$.  We now can express the
gradient operator $\nabla$ in the curvilinear coordinates. Refer to Proposition \eqref{lem:grad}
in the appendix.

\section{Quadratic Approximation of quasi-potential }
In this section, we study the quadratic approximation of the
\qp $V$ near the limit cycle $\Gamma$.  We shall derive a
periodic Riccati differential equation (PRDE) on the limit
cycle and discuss its theoretic properties and numerical
calculations.

\subsection{Asymptotic analysis}
We choose $\tau$ as the physical time in parametrizing the
limit cycle $\Gamma=\set{\gamma(\tau): 0\leq \tau \leq \T}$.
The other choice, such as the arc-length parametrization,
can be easily transformed from the time-parametrization.
The main idea is to write $V(x)$ in terms of the curvilinear
coordinate $(\tau,z)$ near $\Gamma$ and apply the Taylor
expansion of $V$ around $\Gamma$.

Firstly, we have $V(\tau,z)\geq 0$ and particularly on the
limit cycle $V(\tau,z=0)=0$. Therefore
$\partial_{z^i} V(\tau,0)\equiv 0$,
$\forall~ 0\leq i\leq d-1$.  Consequently, the Taylor
expansion of $V(\tau, z)$ in $z$ at $z=0$ reads
\begin{equation}
  \label{eqn:Vapp}
  \begin{split}
    V(\tau,z)
&=\frac12\sum_{i,j=1}^{d-1}z^iz^jG_{ij}(\tau)+O(\abs{z}^3)=\frac12 {z}^\tr {G(\tau)z}+O(\abs{z}^3),
\end{split}
\end{equation}
where the $(d-1)\times (d-1)$ symmetric matrix $G(\tau)$ is
$G_{ij}(\tau)=\partial^2_{z^iz^j}V(\tau,0)$, which is to be
determined.  Since
\[
\partial_\tau V=\frac12\inpd{z}{\dot{G}(\tau)z}+O(\abs{z}^3)
~~~ \mbox{ and } ~~~
\partial_z V=G(\tau)z+O(\abs{z}^2),
\]
then from Proposition \ref{lem:grad} and
$1/(\lambda^{-1}+z)=\lambda+O(z)$, we have
\begin{equation}\label{eqn:dV}
  \begin{split}
    \nabla V=~&\lambda(\tau)\left(\inpd{z}{\left[\frac12\dot{G}(\tau)-\tomega(\tau)^\tr G(\tau)\right]z}+O(\abs{z}^3)\right)e^0(\tau)\\
    &+\sum_{i=
      1}^{d-1}\left[(G(\tau)z)_{i}+O(\abs{z}^2)\right]e^{i}(\tau).
  \end{split}
\end{equation}

Secondly, we expand the coefficients $b(x)$ and $a(x)$ in
the equation \eqref{SDE} in terms of the curvilinear
coordinates $(\tau,z)$.  For any $x$ in a neighborhood of
$\Gamma$, we write the drift vector
\begin{equation}\label{eqn:b}
  b(x)=b\Bigl(\gamma(\tau)+\sum_{i=1}^{d-1} z^{i} e_i(\tau)\Bigr)=\sum_{i=0}^{d-1} B^i(\tau,z)e_i(\tau),
\end{equation}
where by \eqref{eqn:ee} the coefficients are
$$B^i(\tau,z):=\inpd{e^{i}(\tau)}{   b(x)},  ~~0\leq i\leq d-1.$$
On the limit cycle, we have
$ b (\gamma(\tau)
)=\dot{\gamma}(\tau)=\lambda(\tau)^{-1}e_0(\tau).  $
It follows that $ B^0(\tau,0)=\lambda(\tau)^{-1}$ and
$B^i(\tau,0)=0$, $1\leq i\leq d-1$. In addition,
\begin{equation} \label{eqn:defJ}
  \partial_{z^j}B^i(\tau,0)
  =\inpd{e^i(\tau)}{\partial_x b(\gamma(\tau))e_j(\tau)},
  ~\quad~1\leq i,j\leq d-1.
\end{equation}
We denote the right hand side of \eqref{eqn:defJ} as the
$(i,j)$ entry, $1\leq i,j\leq d-1$, of the $(d-1)$ by
$(d-1)$ matrix $\tJ(\tau)$.  Note that
$\partial_xb(\gamma(\tau))$ is the original Jacobian matrix
evaluated on the limit cycle $\Gamma$, and $\tJ(\tau)$ may
be viewed as the Jacobian matrix restricted in the
$z$-space.  Therefore in the neighborhood of $\Gamma$, we
have the expansion
\begin{align}
  \label{eqn:beta}
  B^0(\tau,z)&=\lambda(\tau)^{-1}+O(\abs{z}),\\
  B^i(\tau,z)&=\sum_{j=1}^{d-1}z^j\partial_{z^j}B^i(\tau,0)+O(\abs{z}^2)\nonumber\\
             &=(\tJ(\tau)z)^i+O(\abs{z}^2),\quad 1\leq i\leq d-1,
               \label{eqn:Bi}
\end{align}

For the diffusion tensor $a(x)$, the approximation is simply
\begin{equation} \label{eqn:a}
  a(x)=a\Bigl(\gamma(\tau)+\sum_{i=1}^{d-1} z^{i}
  e_{i}(\tau)\Bigr)=a(\gamma(\tau))+O(\abs{z}).
\end{equation}

Then by plugging \eqref{eqn:b},
\eqref{eqn:beta},\eqref{eqn:Bi} and \eqref{eqn:a} into the
Hamiltonian in \eqref{def:H}, we have that
\[
\begin{split}
  \Hmt(x,\nabla V)=~&\inpd{b(x)}{\nabla V}+\frac12\inpd{\nabla V}{a(x)\nabla V}\\
  =~&\inpd{z}{\left[\frac12\dot{G}(\tau)-\tomega(\tau)^\tr
      G(\tau)\right]z}
  +\sum_{i=1}^{d-1}(G(\tau)z)_i (\tJ(\tau)z)^i\\
  &+\frac12\sum_{i,j=1}^{d-1} (G(\tau)z)_i \inpd{e^i(\tau)}{a\bigl(\gamma(\tau)\bigr)e^j(\tau)}(G(\tau)z)_j+O(\abs{z}^3)\\
  =~&\inpd{z}{\left[\frac12\dot{G}(\tau)-\tomega(\tau)^\tr
      G(\tau)\right]z}
  +\inpd{G(\tau)z}{\tJ(\tau)z}\\
  &+\frac12\inpd{G(\tau)z}{\tA(\tau)G(\tau)z}+O(\abs{z}^3),
\end{split}
\]
where $\tA$ is the positive definite symmetric matrix whose
element in the $i$-th row and $j$-th column is given by
\[
\inpd{e^i(\tau)}{a\bigl(\gamma(\tau)\bigr)e^j(\tau)},
~~1\leq i, j\leq d-1.
\]
$\tA$ is the restriction of the original diffusion matrix
$a$ in the normal plane $P$.  Hence, by \eqref{eqn:H},
equating the terms with the same order $z^iz^j$ yields the
following periodic Riccati differential equation (PRDE)
\[
\dot{G}(\tau)-\tomega(\tau)^\tr
G(\tau)-G(\tau)\tomega(\tau)+G(\tau)\tJ(\tau)+\tJ(\tau)^{\tr}G(\tau)
+G(\tau)\tA(\tau)G(\tau)=0,
\]
i.e.,
\begin{equation}\label{eqn:Riccati}
  \dot{G}(\tau)=
  -\tM(\tau)^{\tr} G(\tau)-G(\tau)\tM(\tau)
  -G(\tau)\tA(\tau)G(\tau),
\end{equation}
where
\begin{equation}\label{eqn:M}
  \tM(\tau):=\tJ(\tau)-\tomega(\tau),
\end{equation}
is the matrix of size $(d-1)\times (d-1)$.  Note that the
coefficients $\tM(\tau)$ and $\tA(\tau)$ are both
$\T$-periodic.  We need to seek a periodic positive definite
solution $G(\tau)$ to \eqref{eqn:Riccati}.

When $G$ is found, the \qp at a point $(\tau,z)$ then can be
approximated locally by the quadratic form
\[
Q(\tau,z):= \frac12 \inpd{z}{G(\tau)z}
\]
due to \eqref{eqn:Vapp}.  Then the momentum $p=\nabla V$ is
approximate by $\nabla Q$ when $\abs{z}\ll 1$ as follows due
to \eqref{eqn:dV} and Proposition \ref{lem:grad}:
\begin{align}
  p(\tau,z)&=\nabla V(x)\approx \nabla Q\nonumber\\
           &=\frac{\abs{e_0(\tau)}}{\abs{\dot{\gamma}(\tau)}}
             \inpd{z}{\left[\frac12\dot{G}(\tau)-\tomega(\tau) ^\tr G(\tau)\right]z}e^0(\tau)
             +\sum_{i=1}^{d-1} \left[(G(\tau)z)_i\right]e^{i}(\tau).\label{eqn:qapp}
\end{align}
This may serve as the initial condition for the Hamiltonian
system.  We can build the MAM by restricting the initial
point on the contour near the limit cycle
$\set{(\tau,z): Q(\tau,z)= \delta}$ with a small positive
$\delta$.  The details are discussed in Section
\ref{sec:num}.

\section{Periodic Differential Riccati Equation}

This section is devoted to the study of PRDE
\eqref{eqn:Riccati} derived from the previous section.  We
investigate the existence and uniqueness of the positive
definite solution and the relation to the linear stability
of the limit cycle as well as the degeneracy of the
diffusion tensor.

\subsection{Solutions of the Riccati  equation} \label{ssec:R}
We first state a theoretical
result concerning the existence of the positive definite
solution for the Cauchy initial value problem of the PRDE
\eqref{eqn:Riccati}.
\begin{proposition}\label{prop:Riccati init pbm}
  Assume that the initial condition $G(0)=G_0$ is symmetric
  and positive semidefinite.  Then the solution of the PRDE
  \eqref{eqn:Riccati} exists and is symmetric and positive
  semidefinite for all $\tau\geq 0$.  Furthermore, if $G_0$
  is positive definite, then so is $G(\tau)$ for all
  $\tau\geq 0$.
\end{proposition}
For the proof, refer to Proposition 1.1 in \cite{Dieci1994}.
From Proposition \ref{prop:Riccati init pbm}, we immediately
conclude the following result on periodic positive definite
solutions to the PRDE \eqref{eqn:Riccati}.
\begin{corollary}
  Assume $G(\tau)$ is a $\T$-periodic symmetric solution to
  the PRDE \eqref{eqn:Riccati}. Then $G(\tau)$ is positive
  (semi)definite for all $\tau$ if and only if $G(\tau_0)$
  is positive (semi)definite for some $\tau_0 \in [0,~\T)$.
\end{corollary}

Next, we give a sufficient and necessary conditions for the
existence and uniqueness of the periodic positive definite
solution to the PRDE \eqref{eqn:Riccati}.  We start with the
connection between the PRDE and the periodic Lyapunov
differential equation (PLDE).
Assume $G(\tau)$ is nonsingular for all $\tau$ (a sufficient
condition is that $G$ is nonsingular at certain $\tau_0$).
Let $H(\tau):=G^{-1}(\tau)$.  Then by \eqref{eqn:Riccati},
$H(\tau)$ solves the following PLDE
\begin{equation}\label{eqn:inverse}
\dot{H}(\tau)
=\tM(\tau)H(\tau)+H(\tau)\tM(\tau)^{\tr}+\tA(\tau).
\end{equation}
It is easy to verify that the solution of the PLDE
\eqref{eqn:inverse} with the initial condition $H(\tau_0)$
is given by
\begin{equation}\label{eqn:Hsol}
  H(\tau)=\Phi_{\tM}(\tau,\tau_0)H(0)\Phi_{\tM}(\tau,\tau_0)^{\tr}+W(\tau,\tau_0),
\end{equation}
where $\Phi_\tM(\tau,\tau_0)$ refers to the state transition
matrix associated with the $\T$-periodic matrix-valued
function $\wt{M}(\cdot)$ (see Definition \ref{def:Phi}) and
\begin{equation}\label{eqn:W}
  W(\tau,\tau_0):=\int_{\tau_0}^\tau \Phi_\tM(\tau,s)\tA(s)\Phi_\tM(\tau,s)^{\tr}\,\mathrm{d} s.
\end{equation}
%
%
%

Before we state our theorem, we need introduce some
definitions and results from the control theory
\cite{Bittanti1986} of linear periodic ordinary differential
equations.

\begin{definition}\label{def:control}
  A pair $(M(\cdot),N(\cdot))$ of $n\times n$ and
  $n\times m$ real $\T$-periodic matrix-valued functions is
  called controllable if there exists no left eigenvector
  $u$ of the monodromy matrix $\bar{\Phi}_M(0)$ satisfying
  the equation $u\,(\Phi_M(\tau,0))^{-1} N(\tau)=0$ for all
  $\tau\in[0,~\T]$.
\end{definition}
  In a special case that the left
null space of $N$ is zero, then $(M,N)$ is controllable for
any $M$.  Note the fact that the left null space of any real
matrix $N$ is the same as that of $NN^\tr$, then we have the
following useful observation.
\begin{lemma}\label{lem:control}
  A pair $(M(\cdot),N(\cdot))$ of $n\times n$ and
  $n\times m$ real $\T$-periodic matrices is controllable if
  and only if the pair $(M(\cdot),N(\cdot)N(\cdot)^\tr)$ is
  controllable.
\end{lemma}

Our main result is the following theorem rigorously
connecting the linear stability of the limit cycle to the
positive definite solution of the PRDE \eqref{eqn:Riccati},
under the assumption of certain controllability determined
by the diffusion tensor $a$.  The proof is based on some
classical results in \cite{Bolzern1988} and is left in
Appendix \ref{appendix:A}.
\begin{theorem}\label{thm:main}
	The PRDE \eqref{eqn:Riccati}
	admits a unique $\T$-periodic positive definite solution  if and only if the following two conditions hold:
	\begin{enumerate}
		\item[(i)] the limit cycle $\Gamma$ is asymptotically stable;
		\item[(ii)] at some $\tau' \in [0,~\T)$, the pair
          $(\bar{\Phi}_\tM(\tau'),W(\tau'+\T,\tau'))$ is
          controllable.  where $\bar{\Phi}_{\tM}$ is the
          monodromy matrix of $\tM$ and $W$ is defined in
          \eqref{eqn:W}.
	\end{enumerate}
\end{theorem}
\begin{remark}
  It is worthy to note that the controllable condition (ii)
  in Theorem \ref{thm:main} as a sufficient condition
  requires only the existence of $\tau' \in [0,~\T)$ such
  that the {\it constant matrix} pair fixed at this $\tau'$
  is controllable.  The conclusion in Theorem \ref{thm:main}
  still holds if condition ($ii$) is replaced by a stronger
  condition { (ii')}:
  $(\bar{\Phi}_\tM(\cdot),W(\cdot+\T,\cdot))$ is
  controllable, or equivalently { (ii'')}:
  ($\tM(\cdot),\tA(\cdot)$) is controllable.  We can obtain
  a weaker (sufficient) condition here in Theorem
  \ref{thm:main} due to Proposition \ref{prop:control} in
  Appendix \ref{appendix:A}.
\end{remark}
In view of \eqref{eqn:W}, we immediately have the following
corollary.
\begin{corollary}
  Assume the limit cycle $\Gamma$ is asymptotically stable.
	\begin{enumerate}
    \item If there exists $\tau' \in [0,~\T)$, such that
      $\tA(\tau')$ is non-singular, then the PRDE
      \eqref{eqn:Riccati} admits a unique $\T$-periodic
      positive definite solution.
		
    \item If for every point $\gamma(\tau)$ on the limit
      cycle, $\tA(\tau)= 0$, then the PRDE
      \eqref{eqn:Riccati} does not admit a unique
      $\T$-periodic positive definite solution.
	\end{enumerate}
	
  \end{corollary}

  Since $\tA(\tau)$ is the restriction of the positive
  semidefinite diffusion tensor $a(\gamma(\tau))$ on the
  normal plane
  $P(\tau)=(e_0(\tau))^\perp
  =\mbox{span}\set{e^1(\tau),\ldots,e^{d-1}(\tau)}$,
  we have the following  proposition.
  \begin{proposition}\label{prop:A0}
    Let $\mathcal{N}(a(\gamma(\tau)))$ and
    $\mathcal{R}(a(\gamma(\tau)))$ be the null space and the
    range space of $a(\gamma(\tau))$, respectively.
	\begin{enumerate}
    \item The following conditions are equivalent:
		\begin{enumerate}
        \item The matrix $\tA(\tau)$ is non-singular;
        \item
          $ \mathcal{N}\left(a(\gamma(\tau))\right)\cap
          P(\tau)=\set{0}$;
          \item
          Any nonzero vector $\xi$ in $P(\tau)$ is not in the subspace $\mathcal{N}\left(a(\gamma(\tau))\right)$.
           	\end{enumerate}
		
		\item The following conditions are equivalent:
		\begin{enumerate}
        \item The matrix $\tA(\tau)= 0$;
        \item
          $P(\tau)\subset \mathcal{N}\left( a(\gamma(\tau))
          \right) $.
        \item
          $\mathcal{R}\left(a(\gamma(\tau))\right)\subset
          \mbox{span}\set{\dot{\gamma}(\tau)}$.
		\end{enumerate}
      \end{enumerate}
\end{proposition}

 The proof is trivial and skipped.   In Proposition
  \ref{prop:A0}, the condition  1.(c) means that
  any perturbation force  in the normal plane  should  not be nullified  by the
 linear transformation  $\xi \to a {\xi}$ in the perturbed system \eqref{SDE}.
 The heuristic argument of
2.(c) is that when $\xi\to \sigma \xi$ (note
    $\mathcal{R}(a)=\mathcal{R}(\sigma)$) transforms any
random force into the tangent direction of the limit cycle, then it is
impossible to escape the limit cycle.

\subsection{Analytic solution for planar limit cycle}
We end this theoretic section with a specific example for
$d=2$ where the solution can be obtained explicitly.

In the case $d=2$, the limit cycle $\Gamma$ is a curve in
the plane. Let $e_0$ be the unit tangent vector
$\dot{\gamma}/\abs{\dot{\gamma}}$ and $e_1$ be the unit
normal vector $\vec{n}$. It follows that
$\wt{\Omega}\equiv 0$.  The matrix PRDE \eqref{eqn:Riccati}
reduces to a scalar PRDE:
$
\dot{G}(\tau)=-2\tM(\tau)G(\tau)
-\tA(\tau)G(\tau)^2,
$
where
$\tA(\tau)=\inpd{e_1}{a e_1}=\inpd{\vec{n}}{a\vec{n}}\geq 0$
is the diffusion coefficient in the normal direction
$\vec{n}$ of $\Gamma$, and
$\tM(\tau)=\inpd{\vec{n}}{(\partial_x b)\vec{n}}$ is the
Jacobian of $b$ along the normal direction $\vec{n}$.  It is
easy to solve the state transition matrix
$ \Phi_\tM(\tau,0)=\exp\left(\int_{0}^\tau \tM(s)\,\d s
\right) $ and the solution of the Lyapunov equation:
\begin{equation}\label{eqn:H2D}
H(\tau)=H(0)\exp\left(2\int_{0}^\tau \tM(s)\,\d s \right)
+\int_{0}^\tau\tA(s)\exp\left(2\int_s^\tau \tM(\sigma)\,\d \sigma \right)\,\d s.
\end{equation}
The $\T$-periodic solution $H(\tau)$ satisfies $H(0)=H(\T)$
and it follows that
\begin{equation}\label{eqn:H0}
{H(0)}=\frac
{\displaystyle\int_{0}^{\T}\tA(s)\exp\left(2\int_s^{\T} \tM(\sigma)\,\d \sigma \right)\,\d s}
{\displaystyle1-\exp\left(2\int_{0}^{\T} \tM(s)\,\d s \right)}.
\end{equation}
The periodic solution of the PRDE is
$G(\tau)=\frac{1}{H(\tau)}$ with $H$ given by
\eqref{eqn:H2D}.

For this planar case, the condition (i) in Theorem
\ref{thm:main} reads
$
\bar{\Phi}_\tM(0)=\exp\left(\int_{0}^\T \tM(s)\,\d s \right)<1,
~\mbox{
	i.e.,}~
\int_{0}^\T \tM(s)\,\d s<0,
$
and the condition (ii) amounts to $\tA(\tau)\not\equiv 0$.
It is easy to see by \eqref{eqn:H0} that these are exactly
the sufficient and necessary conditions for $H(0)>0$, in
which case we have a unique positive definite solution
$G(\cdot)$.

\section{Numerical  methods}
In this section, we develop the related numerical methods
for the computation of the quasi-potential and the optimal
path escaping from the limit cycle.  To find the limit cycle
$\Gamma$, we apply the Newton-Raphson method
\cite{PraticalAlg}.
The Newton-Raphson method can locate the limit cycle with
arbitrary accuracy and a convergence rate much faster than
integrating the ODE.  After the stable limit cycle is found,
the first issue is how to robustly generate the moving frame
on the limit cycle.

\label{sec:num}
\subsection{Construction of the  frame  vectors }
All coefficients of the PRDE \eqref{eqn:Riccati} for $G$ are
based on the matrix $\tomega$ via the moving frame given by
the basis vectors
$\set{e_j(\tau): 0\leq j\leq d-1, 0\leq \tau\leq \T}$.  To
robustly construct this basis with a good quality is not
quite straightforward in high dimension.

The normalization condition $\abs{e_i(\tau)}\equiv 1$ is
always enforced and the first vector $e_0(\tau)$ is set to
be along the tangent direction:
$e_0:=\dot{\gamma}/\abs{\dot{\gamma}}$.  $d=2$ is a trivial
case where $e_1$ is simply obtained by rotating $e_0$ with
the angle $\pi/2$ in the plane.  For a general $d$, one
could construct a set of orthonormal basis $\set{e_j}$ to
have the Frenet frame, by using the time derivatives of
$\gamma$ as Remark \ref{rem:b} has shown.  But this approach
is not practical for high dimension $d$ since it has the
numerically instability in computing the high order
derivatives $\gamma^{(k)}(\tau)$ for $k$  {large}. Furthermore, it usually leads to the
  vectors
$\set{\yu{\gamma^{(j)}}(\tau): 1\leq j\leq d}$ close to
degeneracy, with a very bad condition number  of the
  resulted matrix $E(\tau)$, even after a low-pass
filtering technique applied to the numerical derivatives.
Consequently the calculation of the inverse $E^{-1}$ or the
Gram-Schmidt orthogonalization becomes impractical even for
$d> 4$, observed from our numerical experiments.  The use of
time derivative of $\gamma$ also suffers from the fact that
at some point the derivative may be zero.  For example, in
the plane, a curve can change the rotation from
counterclockwise to clockwise and as a result, the signed
curvature (the first order derivative of $e_0$) is zero at
the turning point.

We propose a robust construction of the basis by choosing
$ {e_1(\tau),\ldots, e_{d-1}(\tau)}$ for each $\tau$ as the
$d-1$ left-eigenvectors (by excluding the eigenvector
associated with the trivial eigenvalue $1$) of the monodromy
matrix $\bar{\Phi}_{\partial_x b(\gamma)} (\tau)$; refer to
Section \ref{ssec:LSLL}.  The sign of the eigenvectors
$e_j(\tau) $ at $\tau>0$ are chosen to be continuous in
$\tau$ after the direction at $\tau=0$ is fixed.  Then
$e_0(\tau)$ is always orthogonal to other basis vectors,
i.e., the normal plane $P(\tau)$ is
$\mbox{span} \set{e_1(\tau),\ldots,e_{d-1}(\tau)}$ for every
$\tau$.  We do not have that $\set{e_1,\ldots,e_{d-1}}$ are
orthogonal by themselves.  If one uses the
right-eigenvectors instead of the left-eigenvectors, then
the only difference is the loss of the orthogonality of
$e_0$ and $\set{e_1(\tau),\ldots,e_{d-1}(\tau)}$.

Our approach has the computational overhead of solving the
matrix-valued ODE \eqref{eqn:Phi} for {\it each} $\tau$ in
parallel. But this method is robust and produces better
numerical representations of $\wt{\Omega}$ and the
coefficients $\wt{M}$ and $\wt{A}$ for PRDE
\eqref{eqn:Riccati}, which is critical to solve $G$
successfully in the next step.

\subsection{Numerical method for the Riccati equation}
\label{ssec:numR}

We use the iterative method to find the periodic solution
$G$ of the PRDE \eqref{eqn:Riccati}.  Each iteration
$G_n \mapsto G_{n+1}$ simply maps a positive definite matrix
$G_n$ to the solution at time $\T$ of \eqref{eqn:Riccati}
with the initial value $G_{n}$.  This method is equivalent
to integrate the PRDE \eqref{eqn:Riccati} in forward time
for sufficiently long time, i.e., we seek for a stable limit
cycle of the dynamical system \eqref{eqn:Riccati} in the
space of positive definite matrix.  The convergence to the positive definite
periodic solution is guaranteed \cite{Shayman1085} if the initial guess $G_0$
is sufficiently large.  We use a scalar matrix $cI_{d-1}$
with a large $c$ for the initial.

As mentioned before, our basis $\set{e_j}$ is constructed
from the eigenvectors of the monodromy matrix.  For some
examples (see Section \ref{ssec:5D}), the eigenvectors may
not be periodic, but anti-periodic (a  simple analogy is the normal vector of  
the M\"{o}bius band).  The anti-periodic situation needs the
following special technique  of computing $G$ in 
\eqref{eqn:Riccati}.  We assume that the first $d_+$ basis
vectors are periodic while the last $d_-=d-d_+$ vectors are
anti-periodic.  Specifically, we have   one basis set
with the matrix form
\[
E^{+}(\tau)= [e^+_0(\tau),\ldots, e^+_{d-1}(\tau)]
\]
where all $e^+_j(\tau)$ are $C^1$ in $[0,\T)$, but
$e^+_j(\T)=e^+_j(0)$ for $0\leq j\leq d_+-1$ and
$e^+_j(\T)=-e^+_j(0)$ for $d_+\leq j \leq d-1$.  Then we
define the following  basis set
\[
E^{-}(\tau):= [e^-_0(\tau),\ldots, e^-_{d-1}(\tau)]
=E^+(\tau) F
\]
where $e^-_j :=e^+_j$ for $0\leq j\leq d_+-1$ and
$e^-_j:=-e^+_j$ for $d_+\leq j \leq d-1$.  Equivalently to
this definition of new basis set, the $d\times d$ matrix $F$ is
$\begin{bmatrix}
  I_{d_-} & 0 \\
  0 & - I_{d_+}\end{bmatrix}.$
Correspondingly, we can have $\wt{M}^+$, $\wt{M}^-$ and
$\wt{A}^+$, $\wt{A}^-$, based on these two local coordinate
systems $E^+$ and $E^-$, respectively.  We then solve the
following $2\T$-periodic Riccati equation for each of the
above iterations $G_n \mapsto G_{n+1}$:
\begin{equation}
\label{eqn:2T}
\begin{cases}
  \dot{G}(\tau)&= -\hat{M}(\tau)^{\tr}
  G(\tau)-G(\tau)\hat{M}(\tau)
  -G(\tau)\hat{A}(\tau)G(\tau),  ~~0\leq \tau \leq 2\T \\
  G(0)&=G_n,
\end{cases}
\end{equation}
where the coefficients   are defined by
\[
\hat{M} (\tau)=
\begin{cases}
  \wt{M}^+(\tau) &  0\leq \tau \leq \T\\
  \wt{M}^-(\tau) & \T <\tau\leq 2\T
\end{cases},~~
\hat{A} (\tau)=
\begin{cases}
  \wt{A}^+(\tau) &  0\leq \tau \leq \T\\
  \wt{A}^-(\tau) & \T <\tau\leq 2\T
\end{cases}.
\]
$G_{n+1}$ is chosen as $G(2\T)$.  It is not difficult to
show that $G(\tau+\T)=FG(\tau)F$ for $\tau\in [0,\T]$.  So,
$G(\tau+\T)$ and $G(\tau)$ have the same eigenvalues and
their eigenvectors are connected by the elementary matrix
$F$: they are essentially the same solution represented by
the two different local coordinate systems $E^+$ and $E^-$.
In other words, the solution of \eqref{eqn:2T} is
$2\T$-periodic but its eigenvalues are $\T$-periodic.  Any
point $x$ near the limit cycle has two representations
$(\tau, z^+)$ or $(\tau, z^-)$, based on $E^+$ or $E^-$,
with the relation $z^-=F z^+$.  The quadratic approximation
of the quasi-potential at $x$ then also has two equivalent
forms:
$Q(x)=\frac12 \inpd{z^+}{G(\tau)z^+}=\frac12
\inpd{z^-}{G(\T+\tau)z^-}$.

\subsection{Solving Hamiltonian systems to generate extremal trajectories  emitting from limit cycle}
\label{ssec:shoot}
Quite like the case of a stable fixed point, one can solve
the Hamiltonian ODE \eqref{eqn:HODE} near the limit cycle
with a pair of the correct initial value on the Lagrangian
manifold of the Hamiltonian system.  The initial values
should be set on the tangent space of the Lagrangian
manifold emitting from the limit cycle.  Select a tiny value
$h>0$ and set the initial condition
$\phi(t=0):=\gamma(\tau)+\sum_{i=1}^{d-1} z_i e_i(\tau) $
with the parameters $(\tau,z)$ on the tube
$ [0,\T)\times (h\mathbb{S}^{d-1})$.  The corresponding
momentum $ p(t=0) = \nabla_x Q $ is set by \eqref{eqn:qapp}.
Then the quasi-potential $V$ at the position $\phi(t)$ is obtained by
$V(\phi(t))= \frac12 \inpd{z}{G( \tau)z}+ \frac12 \int_0^T
\inpd{a(x(t)p(t))}{ p(t)}\dt$.
One need scan all initials on the tube to generate
Hamiltonian trajectories (instantons) for each of them. So,
this approach is preferred   for the low
dimension $d=2$ or $3$.  But the patterns revealed by these
instantons are insightful for understanding the exit
problem.

%

\subsection{Minimum action method to compute the minimum action path emitting from limit cycle}

We next formulate how to use the MAM to compute the minimum
action path and quasi-potential away from the limit
cycle. The version of MAM we used is gMAM.  There are three
approaches. The most straightforward one is to directly use
any traditional MAM by fixing one end of the path on an
arbitrary location of the limit cycle.  Since the initial
path is constructed as a straight line segment connecting
two end points, the choice of this fixed endpoint and the
initial path becomes very important and usually this method
performs bad with a large error both in the path and in the
quasi-potential. The second approach is to confine the end
on the limit cycle rather than on a specific location, so
that this end point of the path can move along the limit
cycle during the minimization procedure. We abbreviate this
approach to gMAM-LC. Theoretically, as the number of grid
points increases, this approach can find the true path with
infinitely long path.  In gMAM-LC, one needs to use very
long path to get relatively accurate result, since a large
part of the path will spiral around the limit cycle but
contribute very little to action.  We remedy this by
introducing the third approach which uses the quadratic
approximation of the \qp within a small tubular neighbor of
the limit cycle.  After splitting the total action as the
sum of the action from the limit cycle to the neighbor and
the action from the neighbor to the final designated
endpoint, one only needs to compute the second part by the
MAM.  Specifically, we use the following tube with a uniform
small radius $h>0$ as the neighbor.
\[
\A_h=\bigcup_{\tau\in(0,\T]} \set{x \in \Real^d:
  x=\gamma(\tau)+\sum_{i=1}^{d-1} z_i e_i(\tau), \ \
  |z|=\sqrt{\sum_{j=1}^{d-1}z^2_i}\leq h }.
\]
An alternative choice is to use the level set of
$\mathcal{Q}_\delta=\set{(\tau,z): Q=\frac12 z^\tr Gz \leq
  \delta}$
as neighborhood, which gives a slightly different version of
the following constraint problem. So we will not give
implementation details of this choice here.

We use the gMAM to calculate the optimal path escaping the
limit cycle $\Gamma$ and ending at some point $x$ outside of
$\A_\delta$. The constraint is that the initial point
$\varphi(0)$ lies on the surface of $\A_\delta$.  By
considering the local coordinate $(\tau,z)$ of $\varphi(0)$,
we have the following constrained minimization problem:
\begin{align}\label{eq:gMAMLQA}
  V(x)=&\min_{\substack
         { \tau\in[0,\mathcal{T}],~|z|=h,
  \\
  \varphi\in AC([0,1]),~\varphi(1)=x  \\
  \varphi(0)=\gamma(\tau)+\sum_{i=1}^{d-1} z_i e_i(\tau) 		
  }
  } \set{ \hat{S}[\varphi]
  +  \frac{1}{2}\inpd{z}{G(\tau)z}}.
\end{align}
We discretize this optimization problem using a linear
finite element approximation to the path $\varphi$, then
solve it by a standard optimization software, e.g. the
\textsf{fmincon} in \textsf{MATLAB}. To implement the
arc-length parametrization, we add equi-arclength
constraints on the discretized grid points. Denote the
objective function as
$F(\varphi, \tau, z) = \hat{S}[\varphi] +
\frac{1}{2}\inpd{z}{G(\tau)z} $.
After discretization of $\varphi$ using $N$ linear finite
elements, we have
\begin{equation}\label{eq:gMAMLQAaction}
  F_N(\varphi, \tau, z) = \hat{S}_N[{\bm \varphi}_N]
  +  \frac{1}{2}\inpd{z}{G(\tau)z}.
\end{equation}
The discrete gMAM action $\hat{S}_N[{\bm \varphi}_N] $ is given as
\begin{equation}\label{eq:gMAMaction}
  \hat{S}_N[{\bm \varphi}_N] = \sum_{i=1}^N
  \frac{\left|\varphi_{i}-\varphi_{i-1}\right|(\left|b(\varphi_{i})\right|+\left|b(\varphi_{i-1})\right|)}{2}
  -\left\langle \varphi_{i}-\varphi_{i-1},\frac{b(\varphi_{i})+b(\varphi_{i-1})}{2}\right\rangle ,
\end{equation}
where
${\bm \varphi}_N=\left\{\varphi_0, \varphi_1, \ldots,
  \varphi_N\right\}$
and we have assumed that $a$ be identity matrix. The
constraints for discretized problem are
\begin{align*}
& |z|^2 - h^2 = 0,
\\
& \gamma(\tau)+\sum_{i=2}^d z_i e_i(\tau) - \varphi_0 = 0,
\\
& |\varphi_i-\varphi_{i-1}|^2 - |\varphi_i-\varphi_{i+1}|^2 = 0, \quad i=1,\ldots, N-1.
\end{align*}
We abbreviate this gMAM approach with local quadratic
approximation for quasi-potential near limit cycle to
gMAM-LQA.  It is easy to see that gMAM-LC can be regarded as
a special case of gMAM-LQA with $h=0$.  
The adaptive choice of the value of  $h$ is to set 
$h\propto 1/N$ to get a second order convergence rate in
$N$.

\section{Numerical examples}

\subsection{Van der Pol oscillator}
The perturbed system takes the form
\begin{equation}\label{VdP}
\left\{
\begin{array}{lcl}
  \d{x} &=& y\dt +\sigma_1(x,y)  \sqrt{\eps}\d w_t^x,\\
  \d{y}&=& (-x+(1-x^2)y )\dt + \sigma_2(x,y)\sqrt{\eps}\d w_t^y.
\end{array}
\right.
\end{equation}
The period of the limit cycle $\Gamma$ in the deterministic
dynamics is $6.6633$. 
We consider the following three cases of
the diffusion coefficients $\sigma_1$ and $\sigma_2$:
\begin{enumerate}
	\item[(i)] isotropic noise:   $\sigma_1=\sigma_2\equiv 1$;
	\item[(ii)] degenerate noise:   $\sigma_1\equiv 0, \sigma_2\equiv 1$;
	\item[(iii)] discontinuous coefficients:  $\sigma_1(x,y)=\sigma_2(x,y)=\begin{cases}
	1 & \mbox{ if } x>0; \\
	0 & \mbox{ if } x<0.
	\end{cases}$
\end{enumerate}
The curvilinear coordinate in representing $G$ is the Frenet
frame.  For the case (i), we plot the limit cycle and the
contours of the approximated \qp in the left panel of Figure
\ref{fig:VD1}.  It is observed that although the level sets
outside of the limit cycle are smooth, the level set with a
small value $Q=0.02$ inside the limit cycle presents four
kinks.  The kinks may come from the breakdown of the
asymptotic approximation of the local coordinates or from
the caustics of the \qp inside the limit cycle.  To explore
this issue, we compute the behaviors of extremal paths near
the limit cycle by applying the symplectic integrator for
the Hamiltonian system as dictated in Section
\ref{ssec:shoot}.  The numerical value of Hamiltonian is
checked to be smaller than $10^{-7}$.  The right panel of
Figure \ref{fig:VD1} shows four extremal paths (projected
onto the $\Real^2$ position space ) emitting from the limit
cycle.  Two of them start from an interior position and
spiral in a quite chaotic way whose final destination is
either being trapped inside or leaving the limit cycle.
These observations may suggest the multi-valued and
self-similar features of the \qp inside the stable limit
cycle.  The similar behavior is observed\cite{SmelyPRE1997,McC2005}   for the
time inverted Van der Pol where the limit cycle is unstable
and the focus inside the limit cycle is linearly stable.

\begin{figure}
		\includegraphics[width=0.48\textwidth]{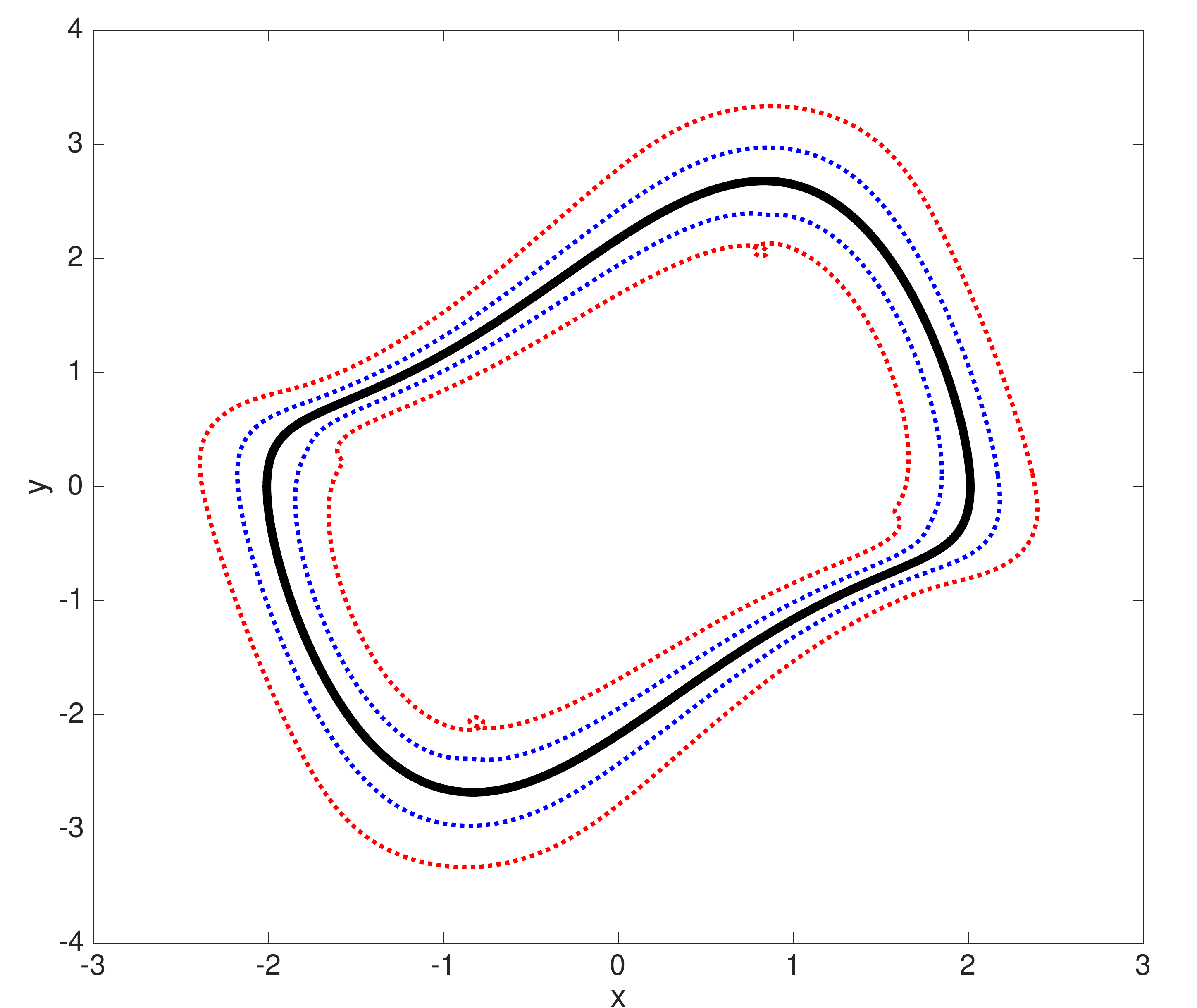}
		\includegraphics[width=0.48\textwidth]{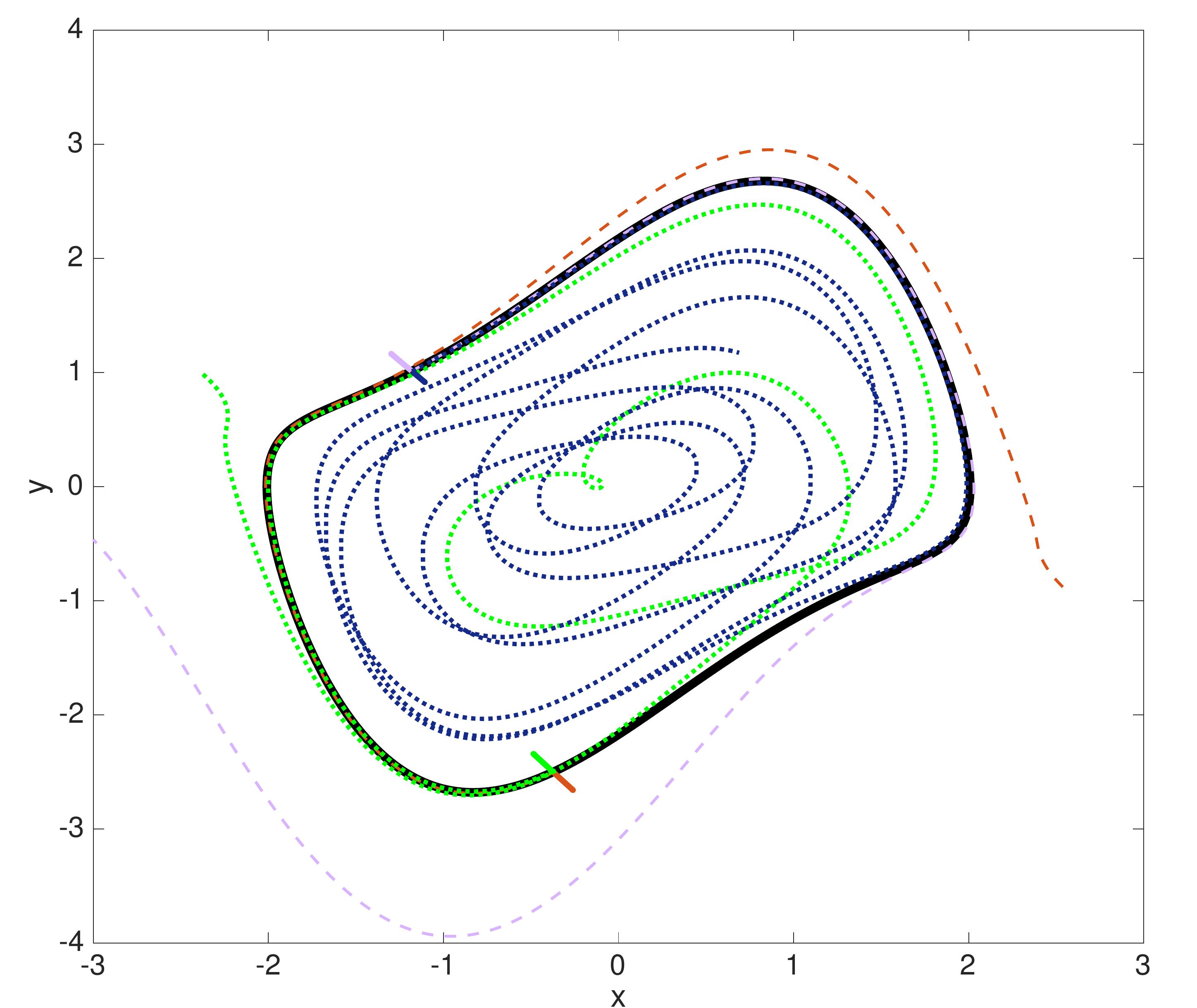}
        \caption{ {\it left}: The limit cycle (thick solid
          curve) and two level sets (dashed curves) of the
          \qp $Q(x)$ corresponding to $Q$ being $0.02$
          (blue) and $0.1$ (red), respectively.  {\it
            right}: The four extremal paths starting from
          four initial positions close to limit cycle.  The
          \qp at these initial positions is as small as
          $10^{-5}$.  Two paths (dotted curves) start from
          the inside of limit cycle (one of them finally
          spirals out of the domain encircled by the limit
          cycle) and the other two paths (dashed curves)
          start outside of the limit cycle.  The four short
          bars indicate the direction of the initial
          momentums associated with the  respective position variables.}
	\label{fig:VD1}
\end{figure}

Next, we study the effect of the diffusion coefficient by
comparing the solution $G$ in the above three cases of
diffusion coefficients.  We plot in Figure \ref{fig:VDG1}
the functions of $G(\tau)$ for the three cases.  As
expected, the \qp in the case (ii) is larger than the \qp in
the case (i) since the dynamics of $x$ component contains no
noise in the case (ii).  For the case (iii), it is seen that
after the diffusion matrix
$a=\mbox{diag}\set{\sigma_1^2,\sigma_2^2}$ is turned off
(the dashed black curve in the figure), the \qp does not
become steep immediately: it becomes significantly large
only after a short period of ``buffering'' region (where
$\tau $ is roughly $2.5\sim 3.5$ in the figure).  In this
region, even without the noise, the deterministic dynamics
on the limit cycle can still carry the previously perturbed
trajectory for a while until the trajectory significantly
leaves the limit cycle.  When the noise is turned on again
we see a sharp decrease of the \qp back to a small value
again. Note that the case (ii) and (iii) have the degeneracy
of the diffusion matrix $a(\tau)$ at certain points, but it
does not affect the existence of the positive solution $G$.

\begin{figure}
	\includegraphics[width=0.68\textwidth]{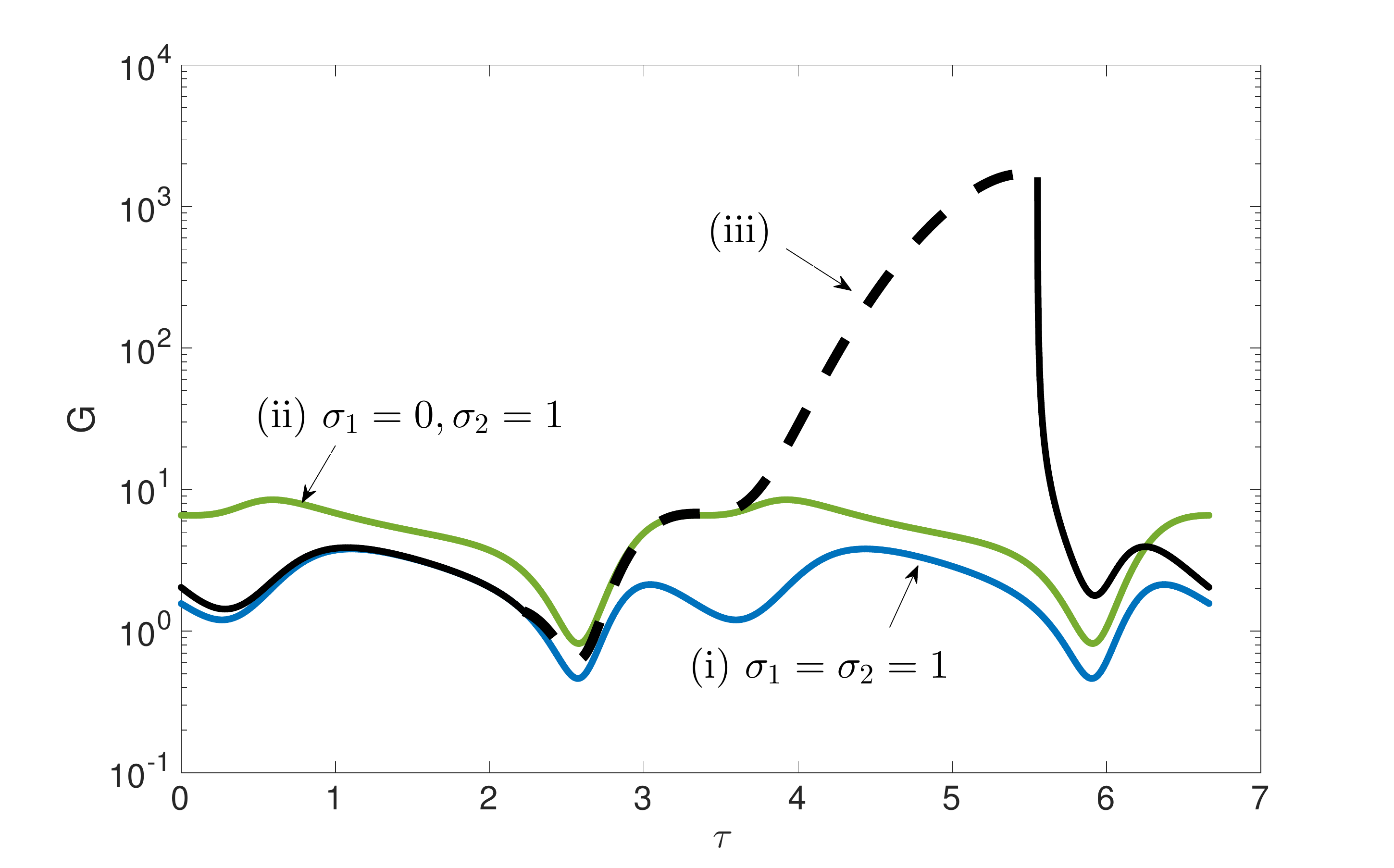}
	\caption{The profiles of $G(\tau)$ v.s $\tau$ for the
      cases (i), (ii), and (iii).  The dashed black segment
      in the case (iii) corresponds to the location where
      $\sigma_1=\sigma_2=0$ (i.e. $x<0$) and the solid black
      segments on the same curve corresponds to the location
      where $\sigma_1=\sigma_2=1$.  }
		\label{fig:VDG1}
\end{figure}

We use gMAM with local quadratic approximation for the
quasi-potential near limit cycle (abbr. gMAM-LQA) and the
gMAM with one end point attached to limit cycle
(abbr. gMAM-LC) to calculate the quasi-potential of point
$(2,-2.5)$ with respect to the given limit cycle.  The
results of the minimum action path is given in
Fig~\ref{fig:ex1MAP}.  The convergence behavior is given
Fig~\ref{fig:ex1conv}, where we use the results of gMAM-LQA
with 160 elements as reference solution. We observe that
while both algorithms have second order convergence, the
gMAM-LQA give better results than the gMAM-LQA method, that
because the gMAM-LC algorithm spends more grid points near
the limit cycle which contribute less action.

\begin{figure}
	\begin{subfigure}[t]{0.4\textwidth}
		\includegraphics[width=\textwidth]{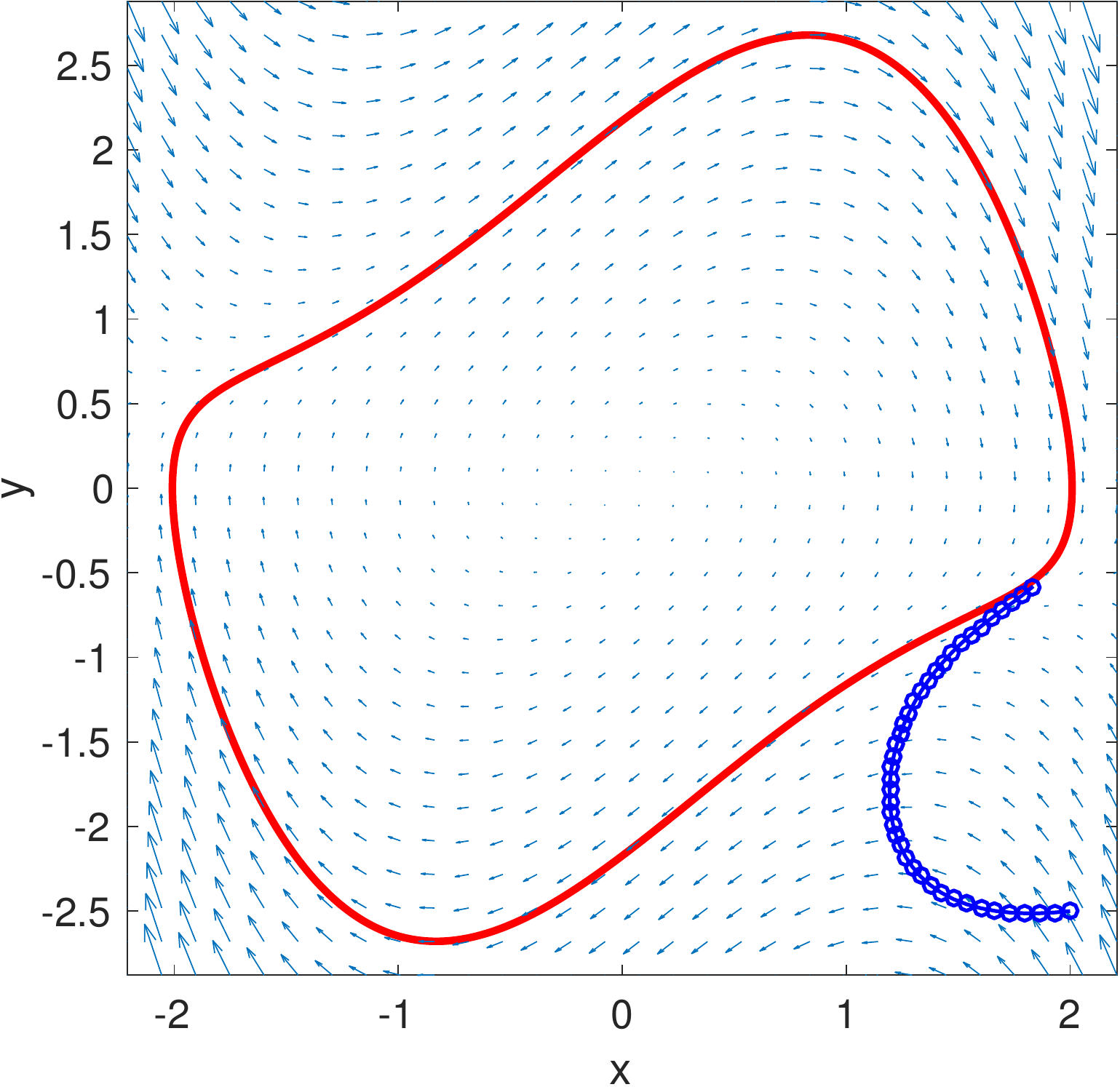}
		\caption{The minimum action path from the limit cycle to point $(2, -2.5)$.}
		\label{fig:ex1MAP}
	\end{subfigure}\hfill
	\begin{subfigure}[t]{0.4\textwidth}
		\includegraphics[width=\textwidth]{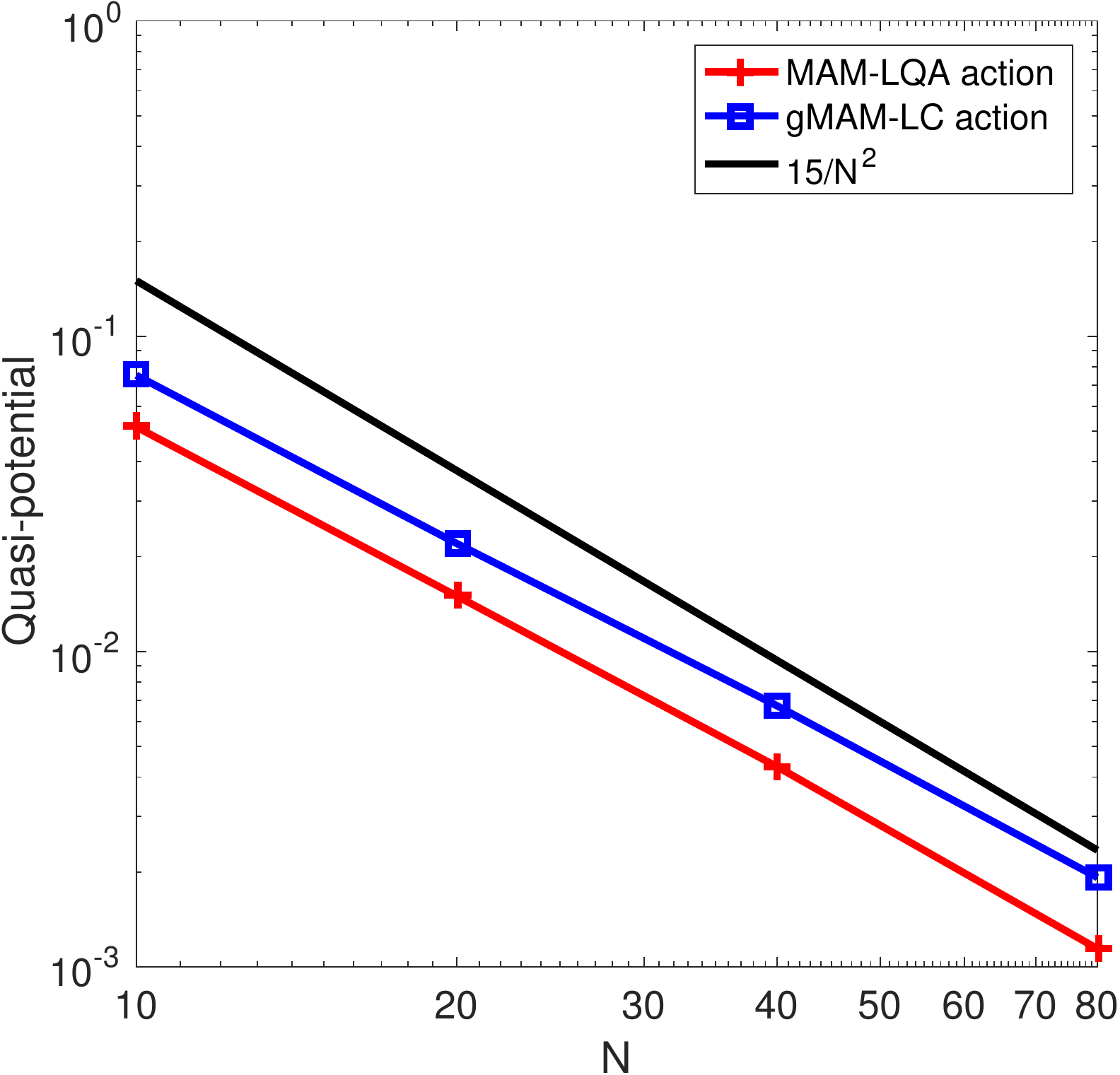}
		\caption{The convergences of the MAM-LQA algorithm and gMAM-LC algorithm.}
		\label{fig:ex1conv}
	\end{subfigure}
	\caption{Numerical results of the minimum action methods for the Van der Pol example.}
\end{figure}

\subsection{ A planar example}
The second example in 2-D is the following system
\cite{Cameron:2012physicaD}:
\begin{equation}\label{twolc}
  \left\{
\begin{array}{lcl}
  \d{x} &=& (x-x^3/3+y-y^3/9)\dt +\sqrt{\eps}\d w_t^x\\
  \d{y}&=& (x+0.9)\dt +\sqrt{\eps}\d w_t^y.
\end{array}
\right.
\end{equation}
This example has two stable limit cycles, separated by a
saddle point located at $(-0.9, 0.6942)$ and the stable and
unstable manifolds of this saddle point.  The \qp for this
example was calculated by directly solving the 2-D
Hamilton-Jacobi partial differential equation
\cite{Cameron:2012physicaD}.  We only compute the \qp around
one of the limit cycle on the top whose period is $5.7966$
and use the MAM to calculate the minimal escape action for
the transition toward the other limit cycle on the bottom by
selecting the final point at the saddle point.  Figure
\eqref{fig:ex2G} plots the  {$G(\tau)$ and MAP}.
The minimal action calculated in the previous work \cite
{Cameron:2012physicaD} is 0.1567. The result of gMAM-LQA
algorithm with $N=160$ is 0.1599 with   a relative error
about $0.1\%$. The minimum action paths obtained using gMAM
and gMAM-LQA with different $N$'s are given in
Fig~\ref{fig:ex2MAP}, in which we observe that the gMAM is
trapped in a local minimum.

\begin{figure}
  \begin{subfigure}[b]{0.49\textwidth}
    \includegraphics[width= \textwidth,height=0.8\textwidth]{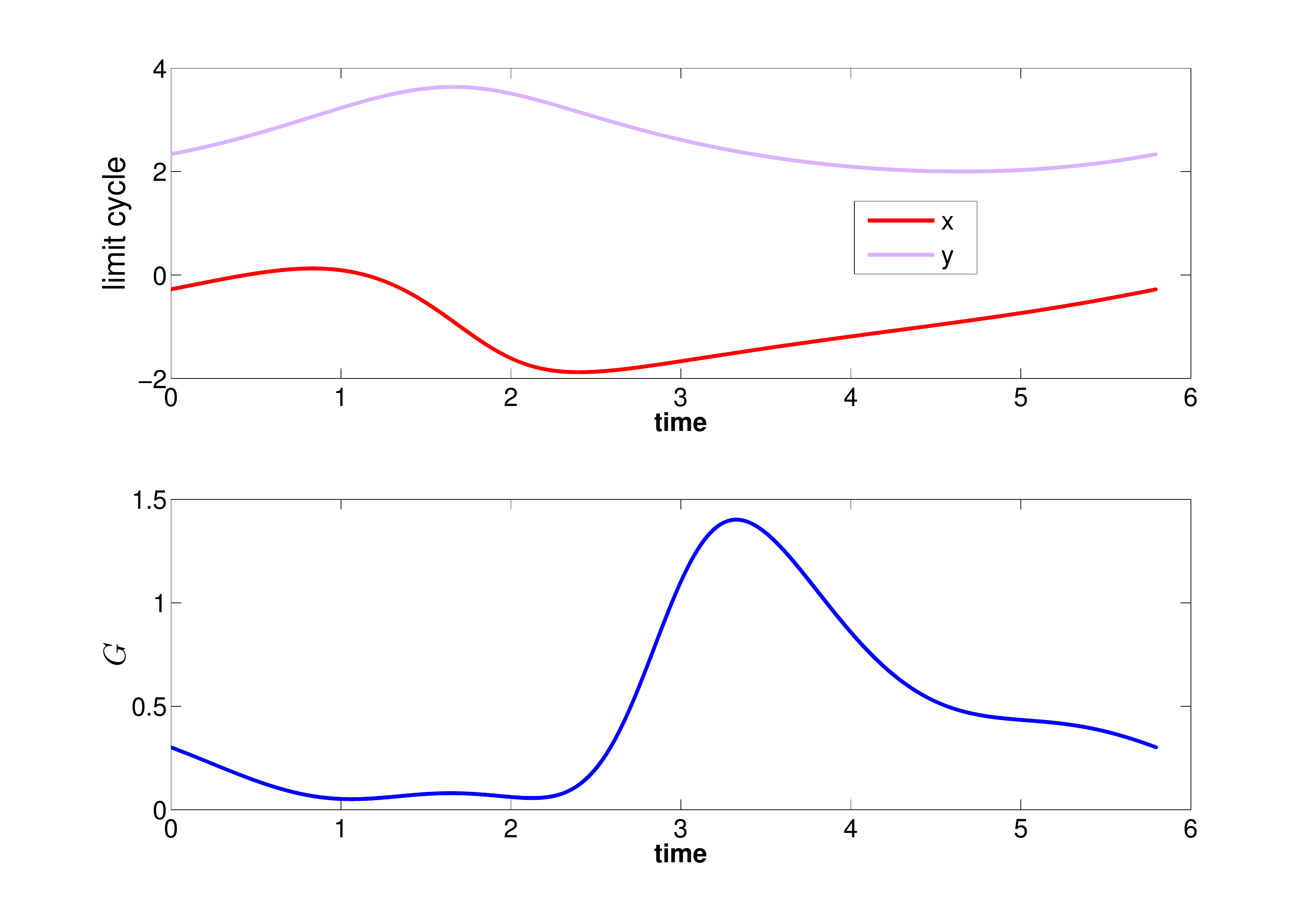}
  \end{subfigure}\hfill
  \begin{subfigure}[b]{0.47\textwidth}
    \includegraphics[width=\textwidth, height=0.76\textwidth]{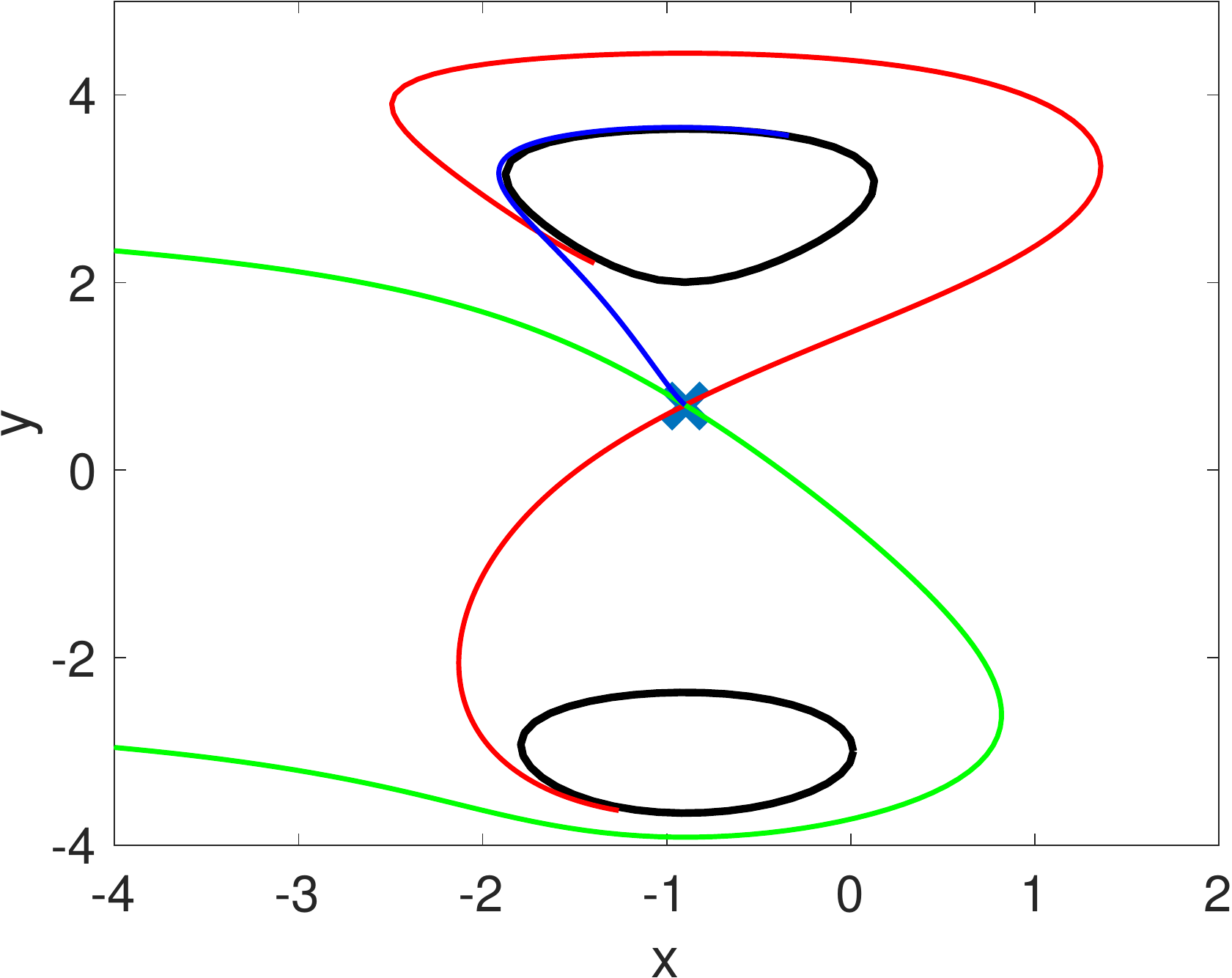}
  \end{subfigure}
  \caption{{\it left}: The profile of $G(\tau)$,
    $0\leq \tau \leq \T$, for the 2D example
    \eqref{twolc}.  The $x$ and $y$ components of the
    limit cycle $\Gamma$ is also shown.  {\it right}:
    The numerical MAP (blue curve) from the limit
    cycle (black) to the saddle point (the marker
    ``$\times$''). The red and yellow curves are,
    respectively, the unstable and stable manifolds of
    the saddle point.}
  \label{fig:ex2G}
\end{figure}

\begin{figure}[htb]
  \begin{subfigure}[b]{0.44\textwidth}
    \includegraphics[width=\textwidth]{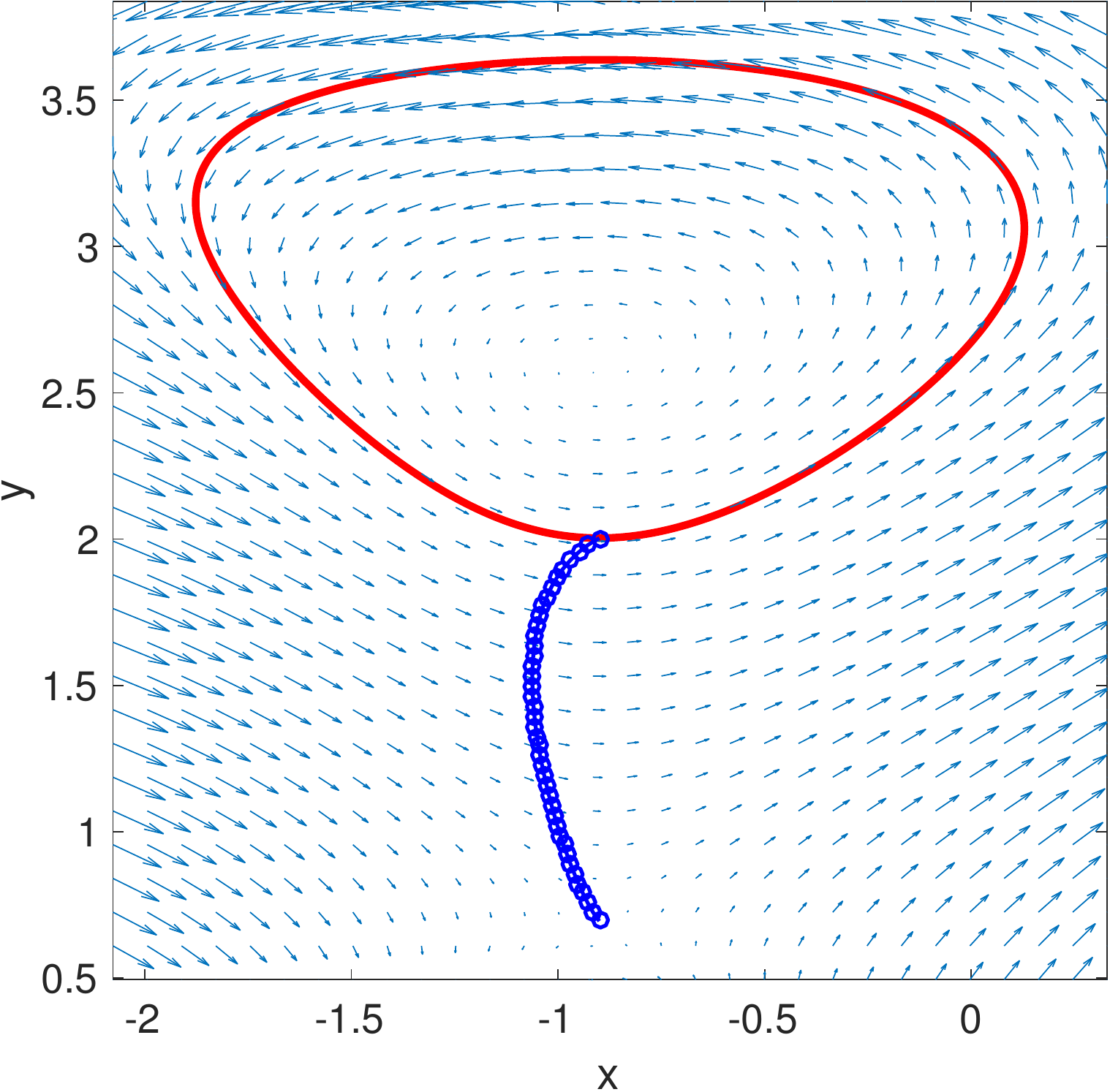}
    \caption{gMAM with initial point fixed on the limit cycle using 40  elements.}
    \label{fig:ex2MAPg40}
  \end{subfigure}
  \hfill
  \begin{subfigure}[b]{0.44\textwidth}
    \includegraphics[width=\textwidth]{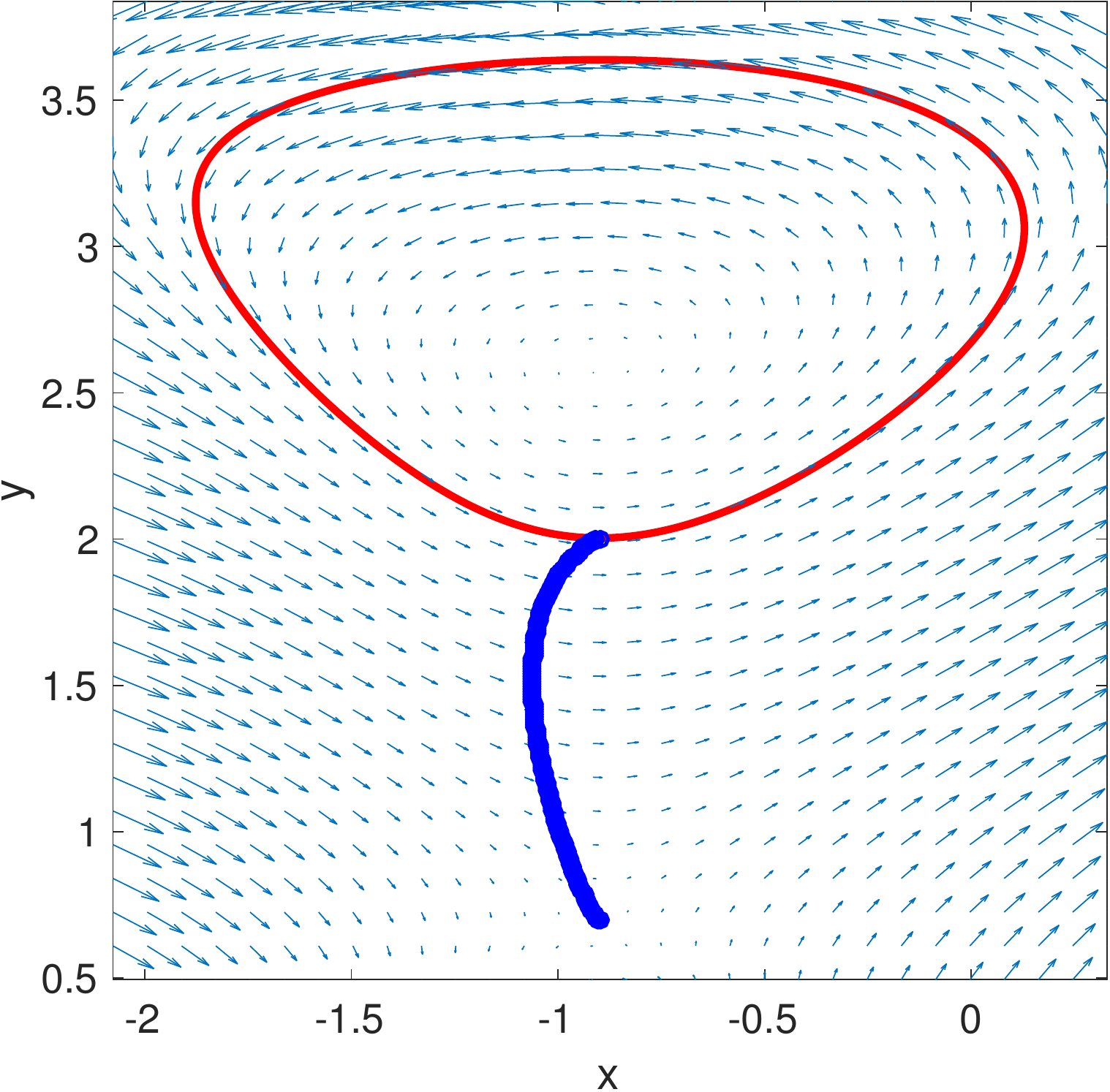}
    \caption{gMAM with initial point fixed on the limit cycle using 160 elements.}
    \label{fig:ex2MAPg160}
  \end{subfigure}
  \vfill
  \begin{subfigure}[b]
    {0.44\textwidth}
    \includegraphics[width=\textwidth]{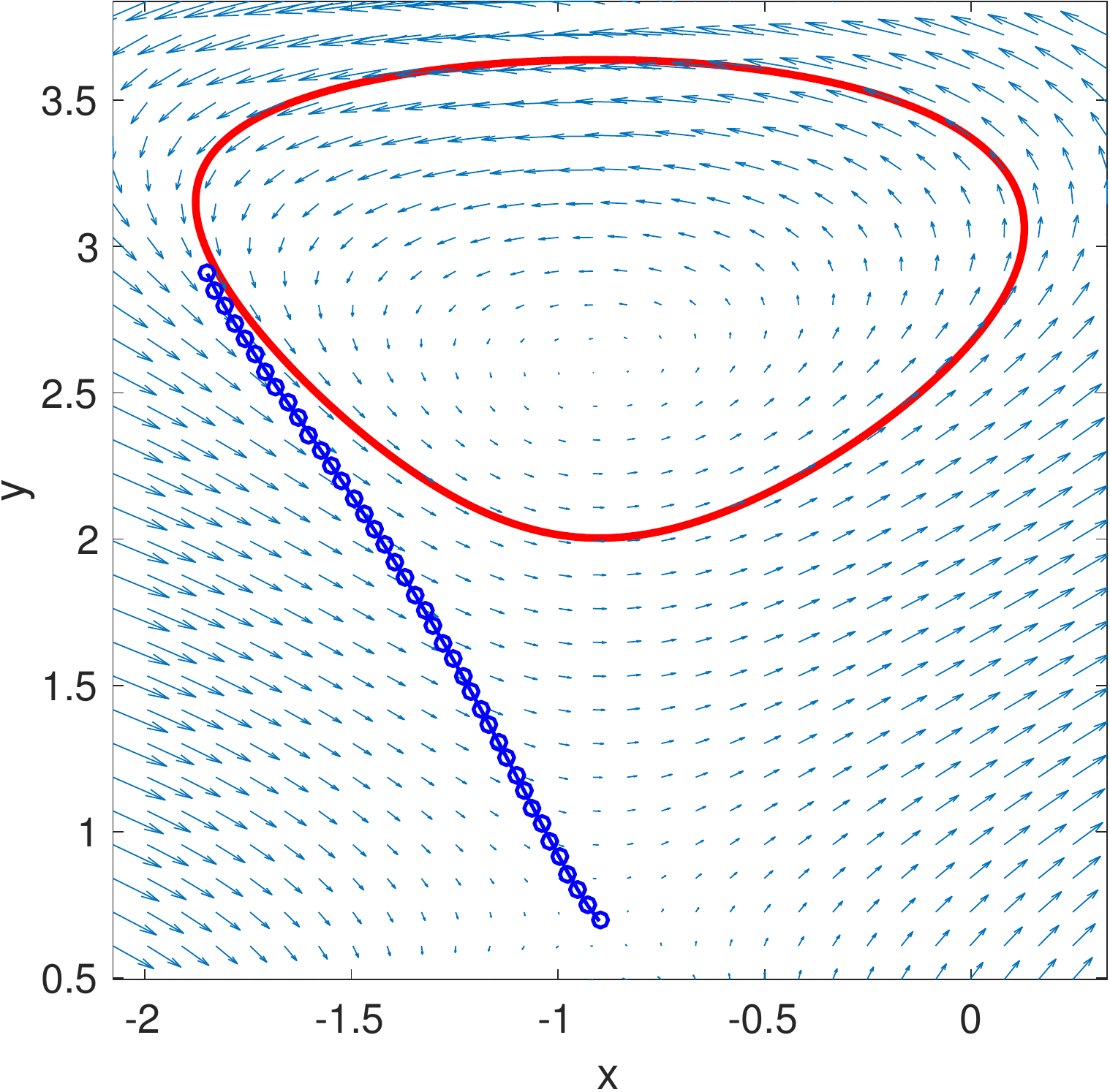}
    \caption{gMAM-LQA using $N=40$ elements.}
    \label{fig:ex2MAPq40}
  \end{subfigure}
  \hfill
  \begin{subfigure}[b]{0.44\textwidth}
    \includegraphics[width=\textwidth]{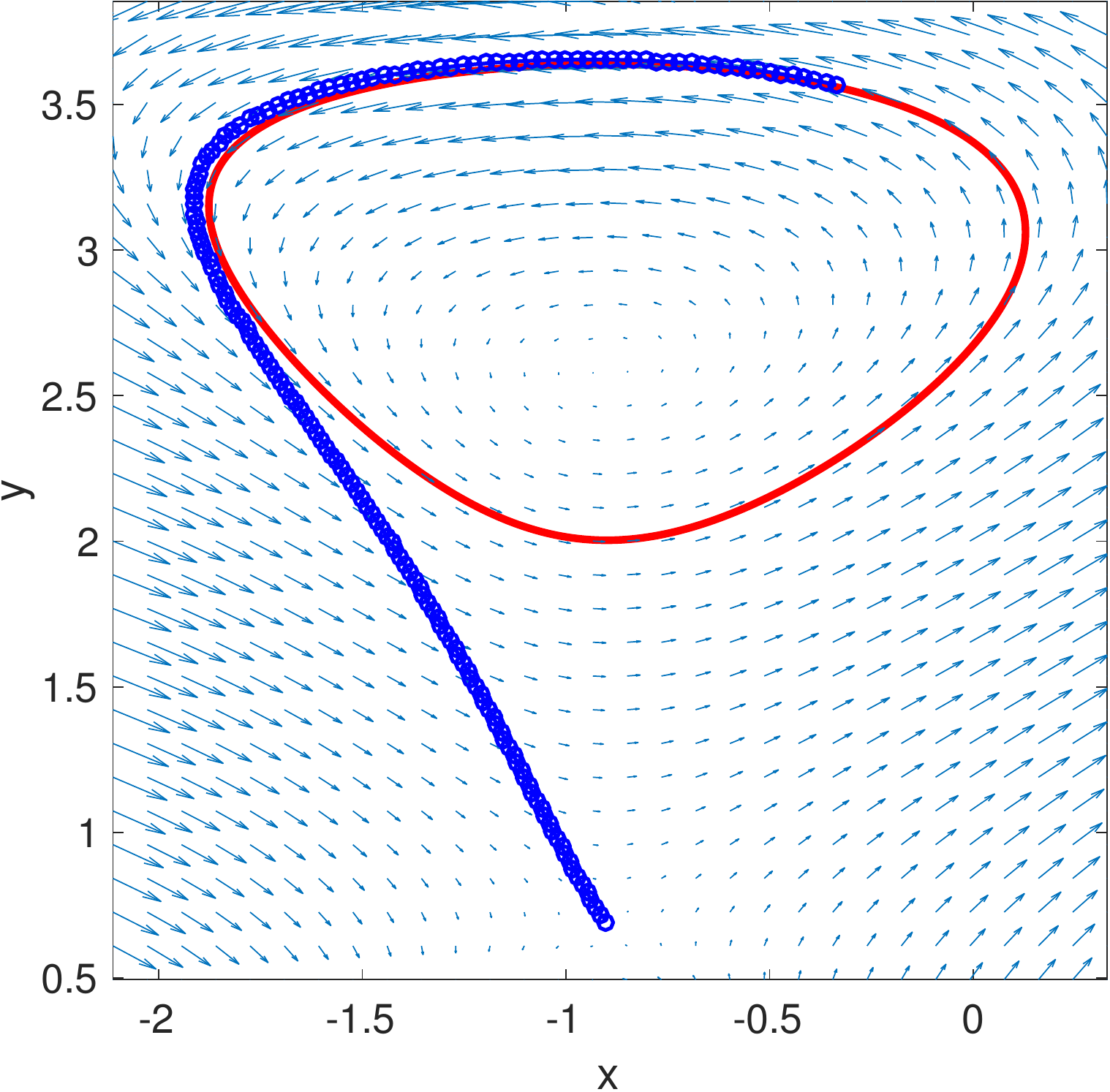}
    \caption{gMAM-LQA using $N=160$ elements.}
    \label{fig:ex2MAPq160}
  \end{subfigure}
  \caption{ The MAP from the limit cycle to the saddle
    point.  The same initial guess is used in all four
    cases: a straight line connecting the two ends in
    (A) and (B).  }
  \label{fig:ex2MAP}
\end{figure}

\subsection{3D Lotka-Volterra  model}
The following randomly perturbed system is the example of 3D
Lotka-Volterra model for   three species
$x=(x_1,x_2,x_3)\in \Real^3$:
\begin{equation}
  \d x_i(t) = x_i  \sum_{j=1}^3 c_{ij}(1-x_j) \dt +  \sqrt{\eps} \d w_i(t).
\end{equation}
We set the matrix $C=[c_{ij}]=
\begin{bmatrix}
  2 &5& 0.5 \\
  0.5 &1 & 1.48\\
  1 & 0.5 &1
\end{bmatrix}.  $
There exist two steady non-negative solutions: $x\equiv 1$
has an one-dimensional stable manifold and an unstable
 {spiral} two-dimensional manifold;
$x\equiv 0$ is unstable and its three unstable eigenvectors
are along three coordinate axes, respectively.  Around the
spiral saddle point $x\equiv 1$, there
is a stable limit cycle $\Gamma$ with the period $6.7965$.
The two eigenvalues of $G(\tau)$ are plotted in Figure
\ref{fig:3D_Geig} and they have a large gap, indicating a
strong directional preference of the \qp near the limit
cycle.  This results in a large aspect ratio in the tube
corresponding to the level set of $Q(x)=\delta $.  Figure
\ref{fig:LV3D} plots this tubular surface for
$\delta =2\times 10^{-5}$, which visually looks like a
ribbon rather than a tube.  The numerical MAP calculated by
gMAM-LQA from the limit cycle to the state $x\equiv 0$
(extinctions of all species) is also shown in
Fig~\ref{fig:LV3D}.

%
%

\begin{figure}
  \includegraphics[width=0.8\textwidth]{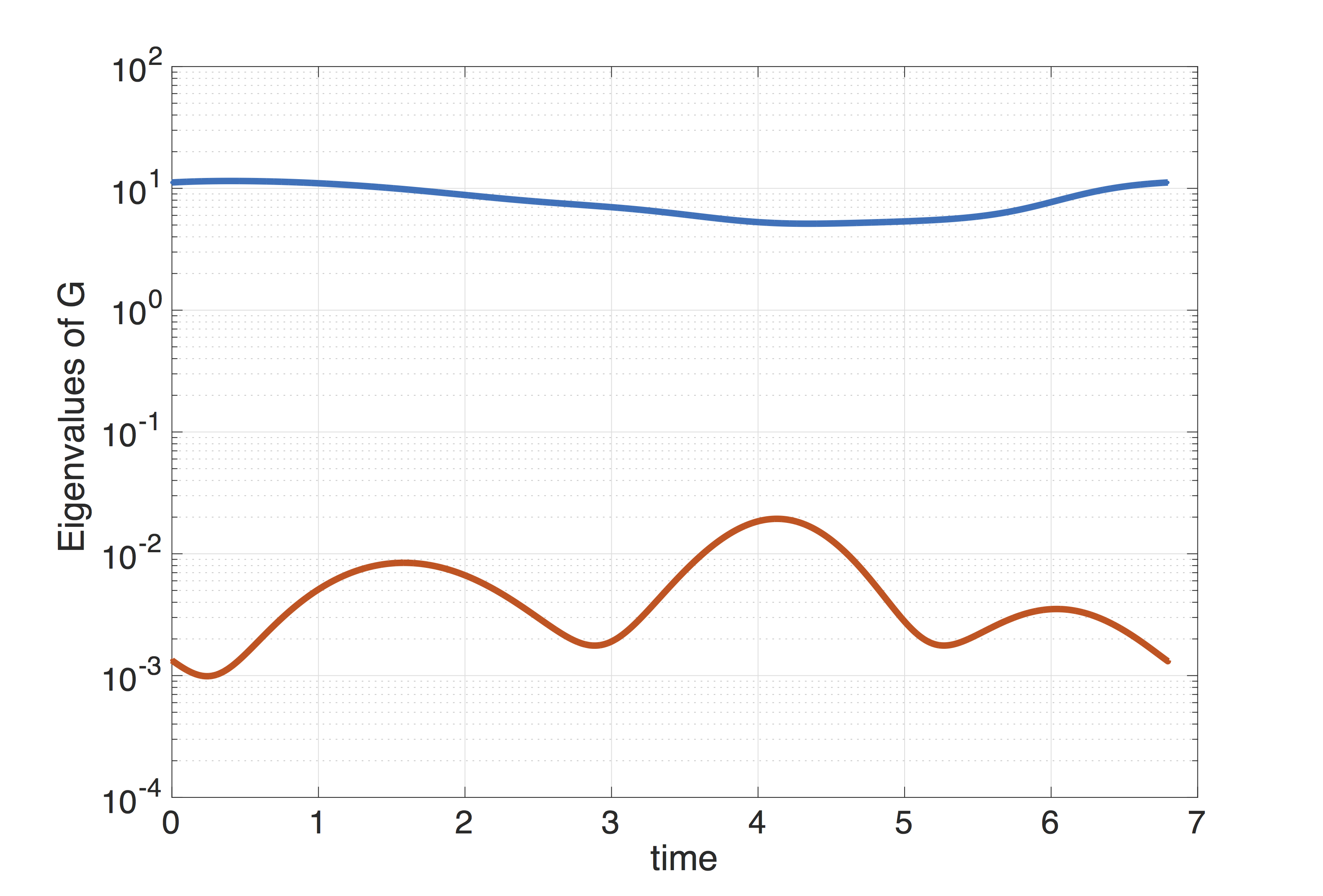}
  \caption{The two eigenvalues of $G$ for the 3D
    Lotka-Volterra model.  }
  \label{fig:3D_Geig}
\end{figure}

\begin{figure}
  \includegraphics[width=0.66\textwidth]{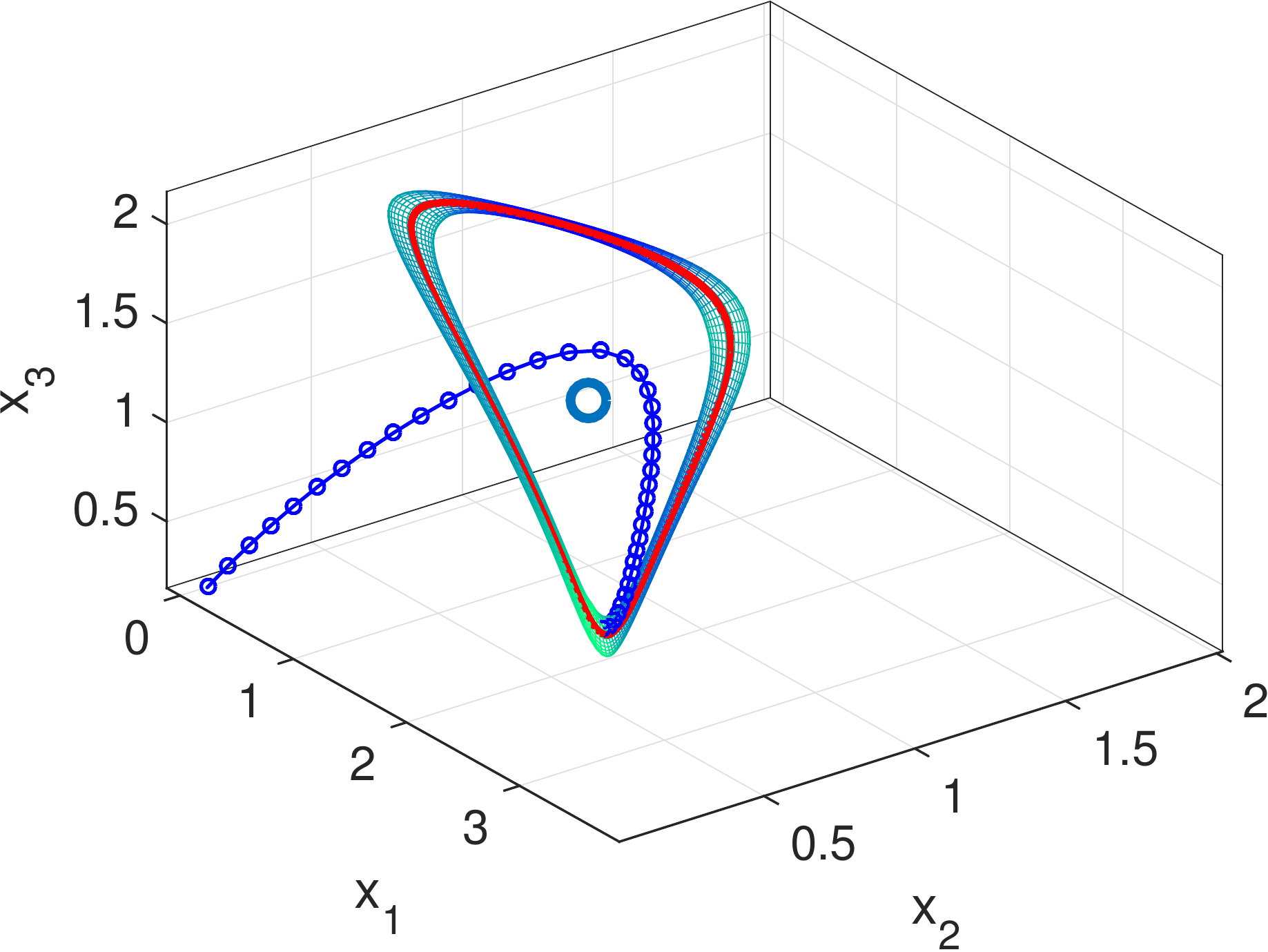}
  \caption{The contour set of the \qp\  for the limit cycle of
    the 3D Lotka-Volterra  model (the  ribbon-like tube in colors).
    The MAP from the limit cycle to the unstable point $x\equiv 0$ are plotted
    as the blue curve with  ``$\circ$'' markers).
    The isolated large circle is the  saddle point $x=(1,1,1)$.
    A red curve is shown to illustrate the flow of the
    deterministic flow.  }
  \label{fig:LV3D}
\end{figure}

\subsection{5D example}
\label{ssec:5D}
%
%

We consider a 5D example whose form is

\begin{equation}\label{fivelc}
  \left\{
    \begin{array}{lcl}
      \dot{x}_1 &=&
                    -x_1+1-2h(x_5)
                    -2h(x_2)+2h(x_2)h(x_5)
                    +2h(x_2)h(x_4)
      \\
                &&\qquad \qquad -2h(x_2)h(x_4)h(x_5)
      \\
      \dot{x}_2 &=&
                    -x_2 +1 -2h(x_1)\\
      \dot{x}_3 &=&
                    -x_3 +1 -2h(x_2)+2h(x_2)h(x_5)
      \\
      \dot{x}_4 &=&
                    -x_4+1-2h(x_3)
      \\
      \dot{x}_5 &=&
                    -x_5+1-2h(x_4)-2h(x_2)+2h(x_2)h(x_4)
\end{array}
\right.
\end{equation}
where $h$ is the sigmoid function
$h(x)=\frac12(1+\tanh (2 x))$.  The corresponding SDE is to
add the isotropic white noise $\sqrt{\eps}\dot{w}_i$ to the
right hand side of \eqref{fivelc}.  The unique limit cycle
is stable with the period $\T= 8.1165$.  When we constructed
the basis $\set{e_0,\ldots, e_4}$ from the left-eigenvectors
of the monodromy matrix, we found that $e_3$ and $e_4$ are
anti-periodic.  So we have to  solve $G$ by
the strategy  based on \eqref{eqn:2T}
which is introduced in Section \ref{ssec:numR}.
Figure \ref{fig:5D} shows the one-period profile of the component $G_{13}$ during the
time interval $(0,2\T)$.   The eigenvalues of $G$ is $\T$-periodic
and we plot the   four eigenvalues of
$G(\tau), \tau\in[0,\T]$ in the same figure.
By using this $G$ and the gMAM-LQA,  we  compute the MAP from the
limit cycle to an endpoint specified at $x= (0.6, 0.6, 0.7, 0.5, 0.3)$.
  Fig~\ref{fig:LV5D} demonstrates the projects of this 5 dimensional path.

\begin{figure}
  \includegraphics[width=0.48\textwidth]{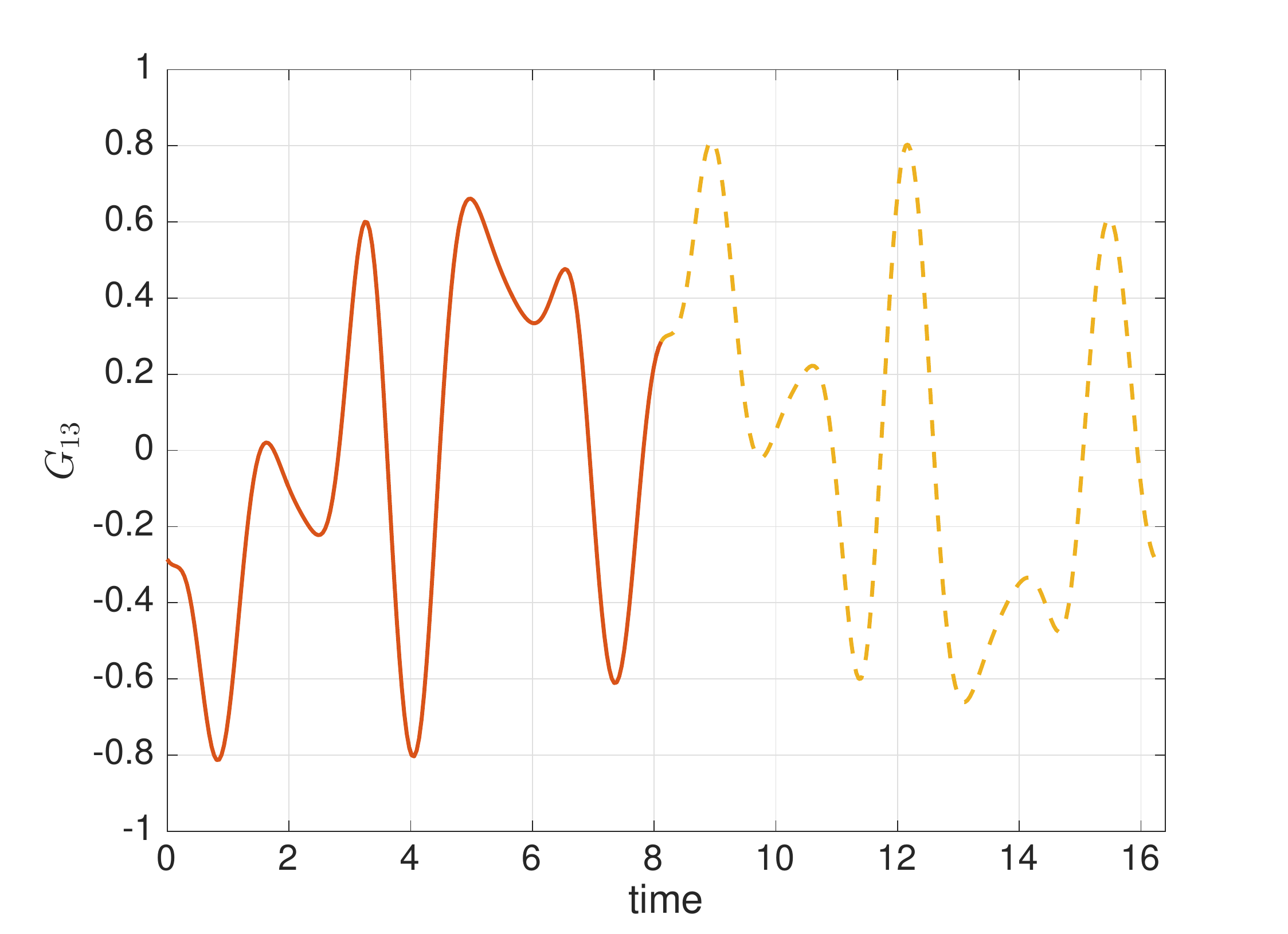}
  \includegraphics[width=0.49\textwidth]{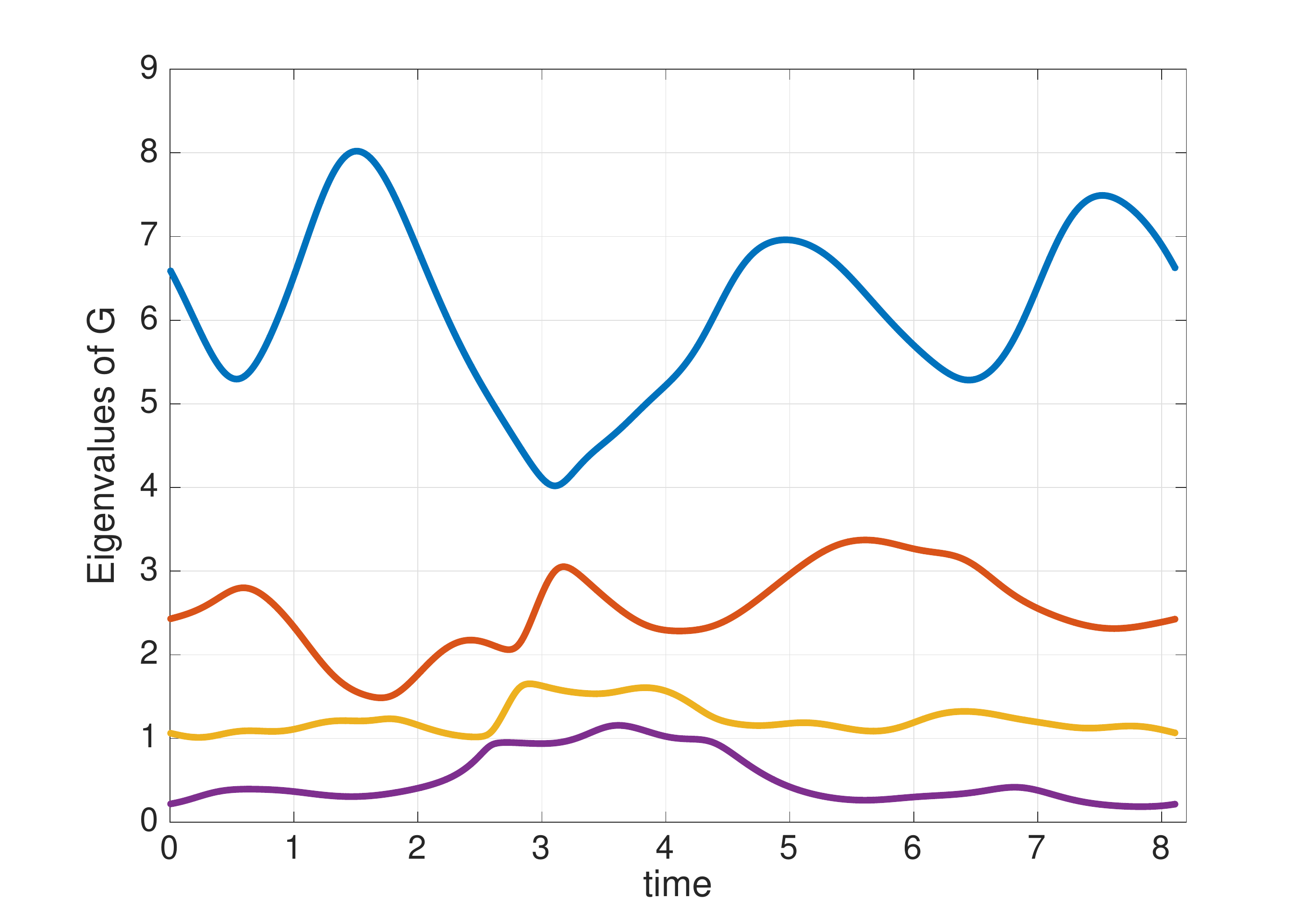}
  \caption{The profile of $(1,3)$ entry of $G(\tau)$,
    $0\leq \tau \leq 2\T$ and the four eigenvalues of $G$.
    $\T=8.1165$.  }
  \label{fig:5D}
\end{figure}

\begin{figure}
	\includegraphics[width=0.48\textwidth]{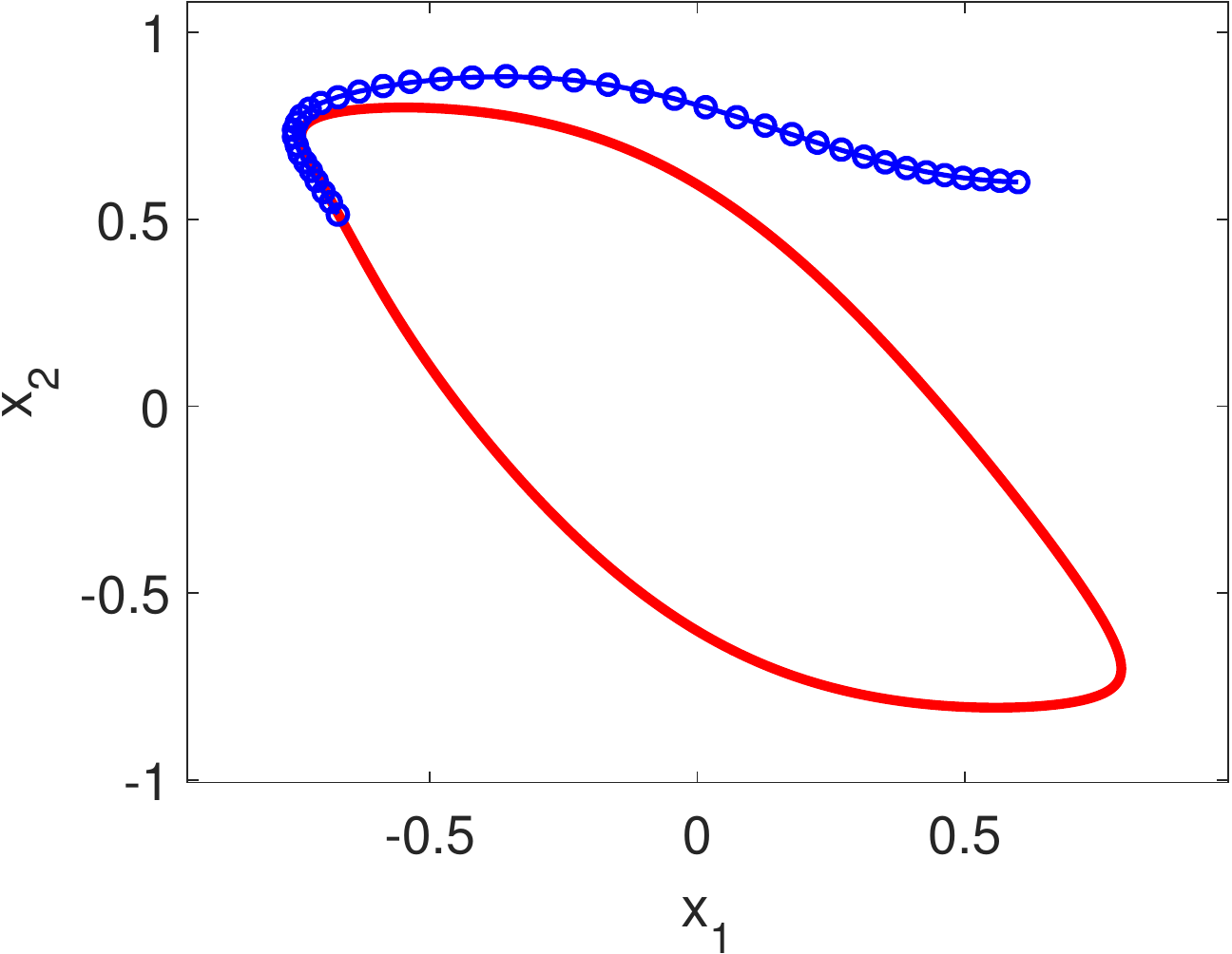}
	\includegraphics[width=0.48\textwidth]{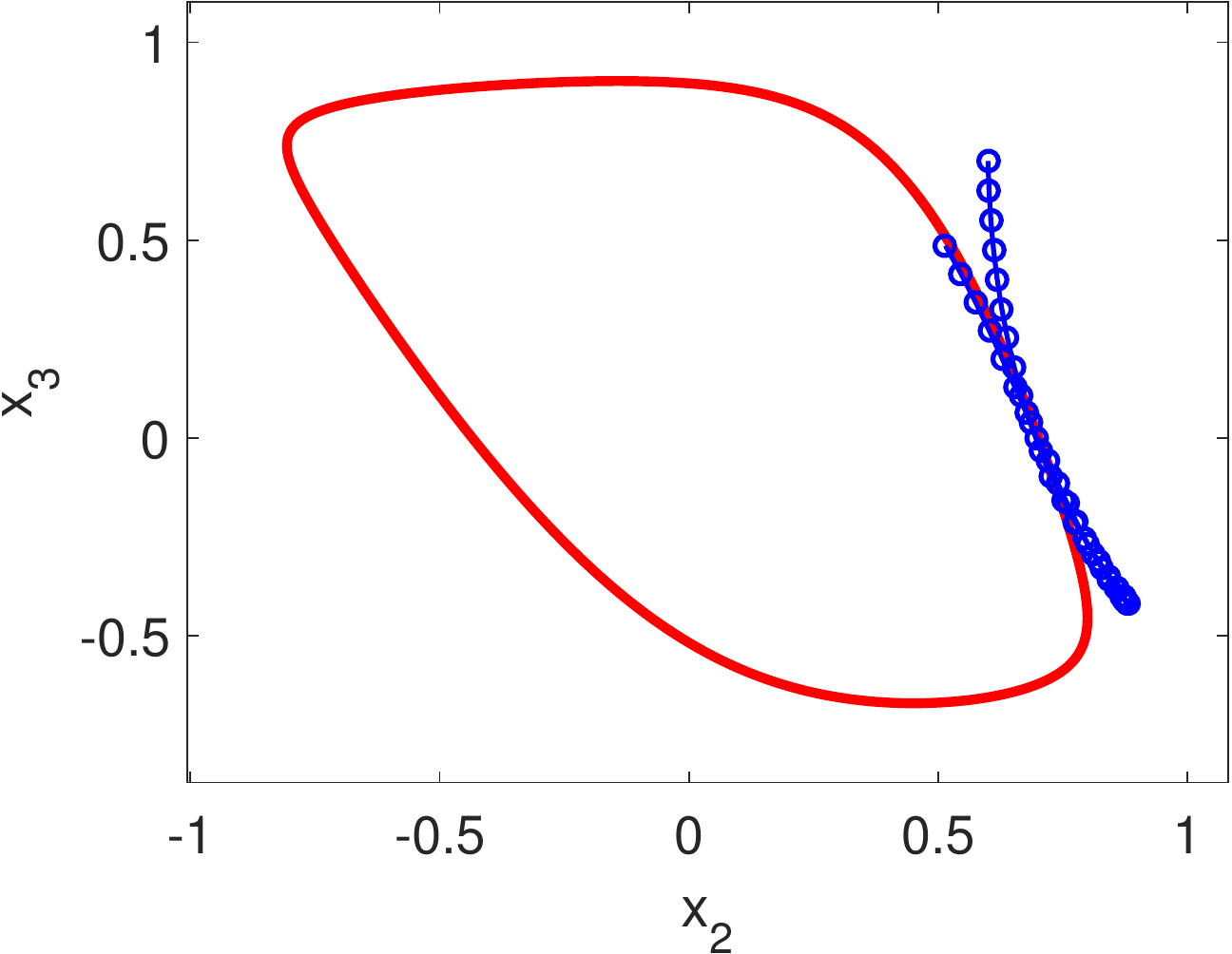}
	\includegraphics[width=0.48\textwidth]{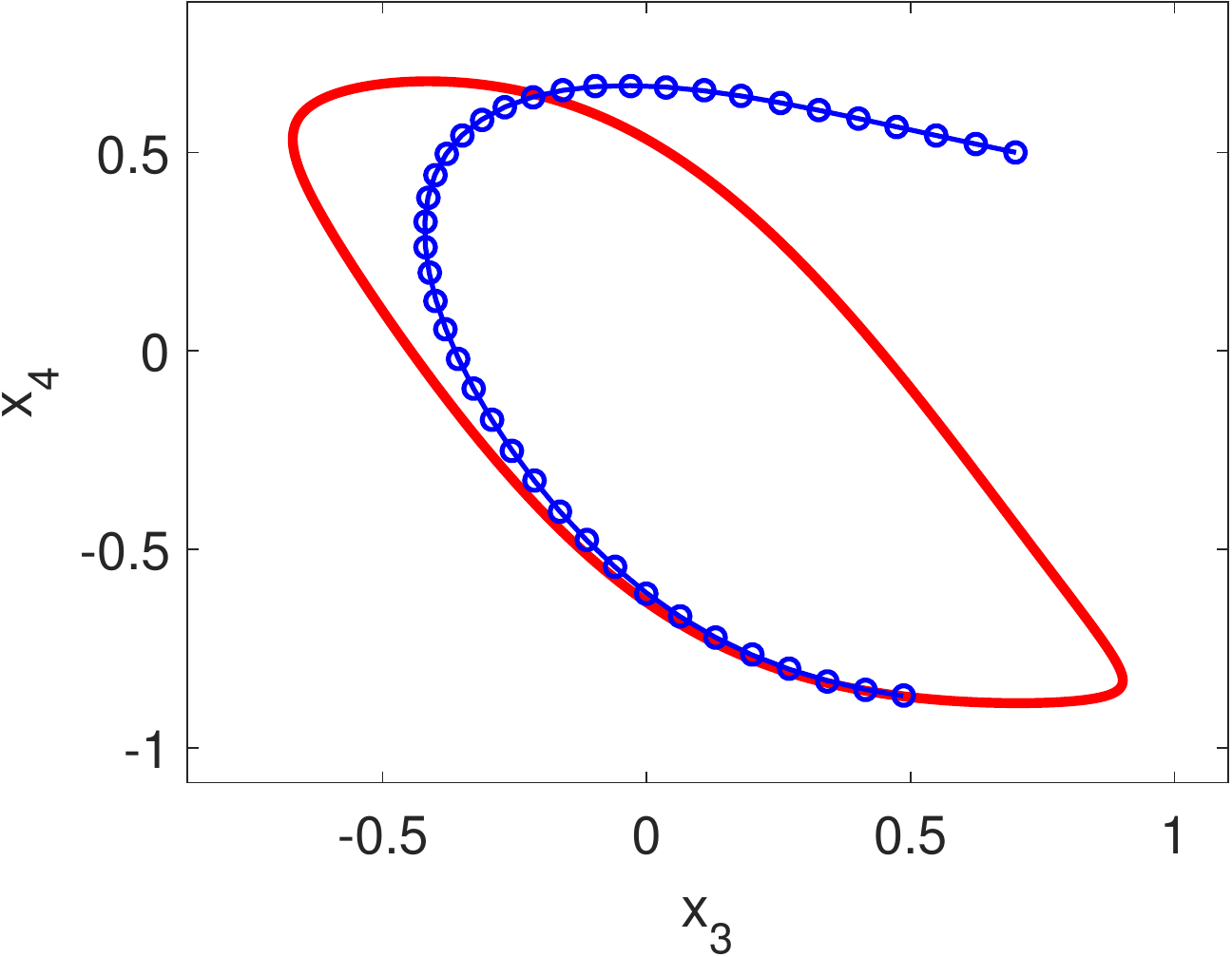}
	\includegraphics[width=0.48\textwidth]{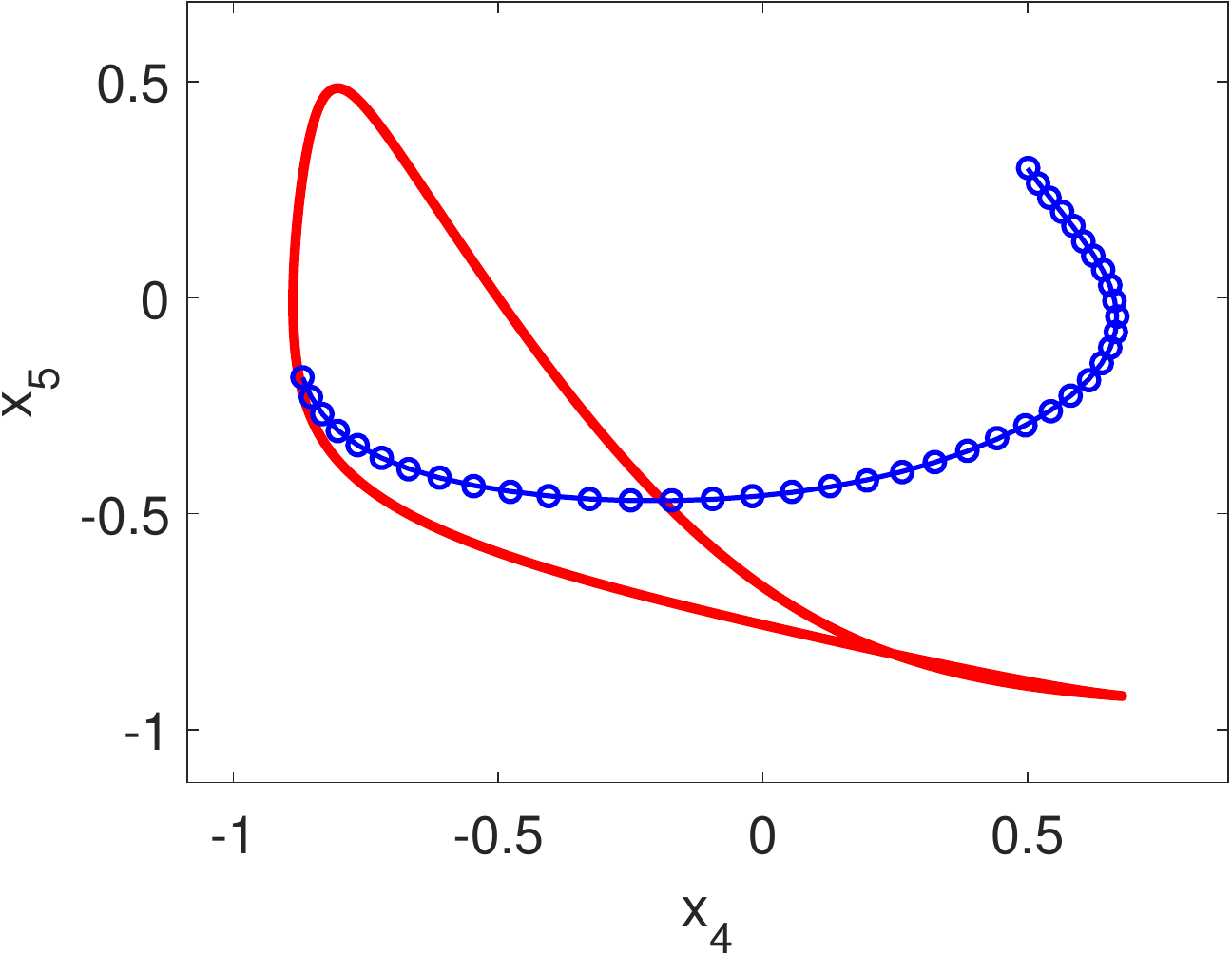}
	\caption{ The MAP from the limit cycle to the point
      $x=(0.6, 0.6, 0.7, 0.5, 0.3)$ calculated using
      gMAM-LQA with $N=40$.  }
	\label{fig:LV5D}
\end{figure}

\section{Conclusion}

We showed how to derive and compute a quadratic approximation of the \qp
near a stable limit cycle in random perturbed systems. The minimum  action method
is improved with the aid of this local approximation  to effectively handle the
infinite length of the optimal path.
With these tools,  one can explore the transition behaviors
related to stochastic oscillatory  in many  high dimensional  problems.

\appendix

\section{}
\label{appendix:A}

\begin{proposition}
  \label{lem:grad}
  The gradient $\nabla f$ of a differential function $f(x)$,
  $x\in\Real^d$, can be written in the curvilinear
  coordinates as
  \[
  \begin{split}
	\nabla f(x)&=  \frac{\partial_\tau f- \sum_{i,j=1}^{d-1}z^j\omega^i_j\partial_{z^i} f}{\lambda(\tau)^{-1}+\sum_{j=1}^{d-1} z^{j}\omega^{0}_j(\tau) }
    e^0(\tau)
    + \sum_{i=1}^{d-1}(\partial_{z^i} f )e^{i}(\tau)\\
	&=\frac{\partial_\tau f-
      \inpd{z}{\tomega(\tau)^\tr\partial_z
        f}}{\lambda(\tau)^{-1}+
      \inpd{e^{0}(\tau)}{\dot{\wt{E}}(\tau) z}} e^0(\tau) +
    \sum_{i=1}^{d-1}(\partial_{z^i} f )e^{i}(\tau),
  \end{split}
  \]
  where $\lambda(\tau)$ is defined in \eqref{eqn:e0},
  $\partial_z f$ denotes the column vector whose components
  are given by $\partial_{z^i} f$.

\end{proposition}

{\it On proof of Theorem \ref{thm:main}:}

It is clear that
a periodic  positive definite solution to the PRDE \eqref{eqn:Riccati} is the inverse of a periodic  positive definite solution to the
PLDE \eqref{eqn:inverse}, and vice versa.
Hence the proof of   Theorem \ref{thm:main} is mainly based on the following classical result on the
PLDE. To state this result, we  need  introduce another type of asymptotical stability for periodic matrix-valued functions.

\begin{definition}\label{def:char multi}
Let $M(\cdot)$ be a $\T$-periodic matrix-valued continuous function. We say that $M(\cdot)$ is  asymptotically stable
if  all the characteristic multipliers of $M(\cdot)$  lie inside the open unit disk.
\end{definition}

The asymptotical stability of $M(\cdot)$ defined above assures that the periodic linear system $\dot{y}(\tau)=M(\tau)y(\tau)$ has
the trivial identically zero solution as an asymptotically stable solution.

In the following proposition, we establish the connection  between the two types of asymptotic stability  in Definition \ref{def:char multi}
and Definition \ref{def:stability of limit cycle}.

\begin{proposition}\label{lem:equivalence}
$\tM(\cdot)$ defined by \eqref{eqn:M} is asymptotically stable if and only if the limit cycle $\Gamma$ of \eqref{eqn:ODE} is asymptotically stable.
\end{proposition}

\begin{proof}
We will proof this proposition by showing that the characteristic multipliers associated with \eqref{eqn:first var}
consist of 1 and the characteristic multipliers associated with $\tM(\tau)$, i.e.,
the eigenvalues of $\bar{\Phi}_{\partial_x b(\gamma(\cdot))}(0)$ consist of 1 and the eigenvalues of $\bar{\Phi}_\tM(0)$.

Let $\Theta(\tau)=E(\tau)^{-1}\Phi_{\partial_x b(\gamma(\cdot))}(\tau,0)E(0)$.
We claim that $\Theta(\tau)$ is a block upper triangular matrix which has the form
$
\Theta(\tau)=\begin{bmatrix}
\Theta_{11}(\tau) & \Theta_{12}(\tau)\\
0 & \Theta_{22}(\tau)
\end{bmatrix},
$
where $\Theta_{11}(\tau)$ is a scalar and is equal to $\lambda(0)/\lambda(\tau)$.
To see this, we note that
the first column of $\Theta(\tau)$ is given by $E(\tau)^{-1}\Phi_{\partial_x b(\gamma(\cdot))}(\tau,0)e_0(0)$,
and the key observation is that
\begin{equation}\label{eqn:observation}
\Phi_{\partial_x b(\gamma(\cdot))}(\tau,0)e_0(0)=\frac{\lambda(0)}{\lambda(\tau)}e_0(\tau).
\end{equation}
To justify \eqref{eqn:observation}, recall that $e_0(\tau)=\lambda(\tau)\dot{\gamma}(\tau)$ and $\dot{\gamma}(\tau)$ solves
the equation \eqref{eqn:first var} of first variation.
Since $\Phi_{\partial_x b(\gamma(\cdot))}(\tau,0)$, by definition, is  the fundamental matrix solution of \eqref{eqn:first var}
and $\Phi_{\partial_x b(\gamma(\cdot))}(0,0)=I$, we infer from the uniqueness of the solution to the initial value problem of \eqref{eqn:first var}
that
\[
\dot{\gamma}(\tau)=\Phi_{\partial_x b(\gamma(\cdot))}(\tau,0)\dot{\gamma}(0),
\]
and \eqref{eqn:observation} follows immediately.

Next, we will show that $\Theta_{22}(\tau)=\Phi_{\tM}(\tau,0)$.
To this end, we will derive the ODE for $\Theta(\tau)$.
Using \eqref{eqn:general FS matrix form}, we have
\[
\frac{\d}{\d\tau}E(\tau)^{-1}=-E(\tau)^{-1}\dot{E}(\tau)E(\tau)^{-1}=-\Omega(\tau)E(\tau)^{-1}.
\]
Recall that  $\Phi_{\partial_x b(\gamma(\cdot))}(\tau,0)$ satisfies
\[
\frac{\partial }{\partial \tau}\Phi_{\partial_x b(\gamma(\cdot))}(\tau,0)=\partial_x b(\gamma(\tau))\Phi_{\partial_x b(\gamma(\cdot))}(\tau,0).
\]
Hence we find
\begin{align}
\dot{\Theta}(\tau)
&=-\Omega(\tau)E(\tau)^{-1}\Phi_{\partial_x b(\gamma(\cdot))}(\tau,0)E(0)+E(\tau)^{-1}
\partial_x b(\gamma(\tau))\Phi_{\partial_x b(\gamma(\cdot))}(\tau,0)E(0)\nonumber\\
&=-\Omega(\tau)\Theta(\tau)+E(\tau)^{-1}\partial_x b(\gamma(\tau))E(\tau)\Theta(\tau)\nonumber\\
&=M(\tau)\Theta(\tau),\label{eqn:ODE for Theta}
\end{align}
where
$
M(\tau):=E(\tau)^{-1}\partial_x b(\gamma(\tau))E(\tau)-\Omega(\tau).
$
Comparing the last equation with \eqref{eqn:M}, we see that
$\tM(\tau)$ is the principal submatrix of $M(\tau)$ by deleting the first row and  the first column.
Partition $M(\tau)$ into blocks that are compatible with the partitions of $\Theta(\tau)$
 and then \eqref{eqn:ODE for Theta} can be written in the blocked form as
\begin{align*}&
\begin{bmatrix}
\dot{\Theta}_{11}(\tau) & \dot{\Theta}_{12}(\tau)\\
0 & \dot{\Theta}_{22}(\tau)
\end{bmatrix}=\begin{bmatrix}
M_{11}(\tau) & M_{12}(\tau)\\
M_{21}(\tau) & \tM(\tau)
\end{bmatrix}\begin{bmatrix}
\Theta_{11}(\tau) & \Theta_{12}(\tau)\\
0 & \Theta_{22}(\tau)
\end{bmatrix}\\
&=\begin{bmatrix}
M_{11}(\tau)\Theta_{11}(\tau) & M_{11}(\tau)\Theta_{12}(\tau)+M_{12}(\tau)\Theta_{22}(\tau)\\
M_{21}(\tau)\Theta_{11}(\tau) & M_{21}(\tau)\Theta_{12}(\tau)+\tM(\tau)\Theta_{22}(\tau)
\end{bmatrix}.
\end{align*}
Since $\Theta_{11}(\tau)=\lambda(0)/\lambda(\tau)\neq 0$, we deduce that $M_{21}(\tau)=0$ and  then it follows that
$
\dot{\Theta}_{22}(\tau)=\tM(\tau)\Theta_{22}(\tau).
$
Also,  we have $\Theta(0)=E(0)^{-1}\Phi_{\partial_x b(\gamma(\cdot))}(0,0)E(0)=I$, so $\Theta_{22}(\tau)$
is a $(d-1)\times(d-1)$ identity matrix. Again, by the uniqueness of the solution to the initial value problem,
we obtain $\Theta_{22}(\tau)=\Phi_{\tM}(\tau,0)$.

Since $\lambda(\tau)$ and $E(\tau)$ are $\T$-periodic, we have
\[
\Theta(\T)=E(\T)^{-1}\Phi_{\partial_x b(\gamma(\cdot))}(\T,0)E(0)=E(0)^{-1}\bar{\Phi}_{\partial_x b(\gamma(\cdot))}(0)E(0),
\]
and
\[
\Theta(\T)=\begin{bmatrix}
\lambda(0)/\lambda(\T) & \Theta_{12}(\T)\\
0 & \Phi_{\tM}(\T,0)
\end{bmatrix}
=\begin{bmatrix}
1 & \Theta_{12}(\T)\\
0 & \bar{\Phi}_{\tM}(0)
\end{bmatrix}.
\]
From the last two equations, we conclude that
 $\bar{\Phi}_{\partial_x b(\gamma(\cdot))}(0)$ and $\Theta(\T)$ have the same eigenvalues since they are similar,
and the  eigenvalues of $\Theta(\T)$ consist of 1 and the eigenvalues of $\bar{\Phi}_\tM(0)$.
This justifies what we asserted and the proof is complete.
\end{proof}

\begin{remark}
In Definition \ref{def:stability of limit cycle},
the asymptotical stability of the limit cycle $\Gamma$ is clearly defined  without using the moving affine frame
$e_i(\tau)$, $0\leq i\leq d-1$. Thus in view of Proposition \ref{lem:equivalence},
the asymptotical stability of $\tM(\cdot)$ is also a property independent of the choice of
the moving affine frame
$e_i(\tau)$, $0\leq i\leq d-1$, although
the matrix $\tM(\cdot)$ defined by \eqref{eqn:M} depends explicitly on  the moving affine frame
$e_i(\tau)$, $0\leq i\leq d-1$.
\end{remark}

Now Theorem \ref{thm:main}
 can be easily derived from  Proposition \ref{lem:equivalence}, Lemma \ref{lem:control} and the following classical results   \cite{Bolzern1988}.
%
%
%

\begin{proposition}[Theorem 20 in \cite{Bolzern1988}]\label{thm:well-posed}
For any $\T>0$, let  $M(\cdot)$, $N(\cdot)$ be  two $\T$-periodic matrix-valued functions
with size $n\times n$ and  $n\times m$ respectively.
Then the following    PLDE
\[
\dot{P}(\tau)=M(\tau)P(\tau)+P(\tau)M(\tau)^{\tr}+N(\tau)N(\tau)^\tr
\]
admits a unique $\T$-periodic positive definite solution $P(\tau)$ if and only if the following two conditions hold:
\begin{enumerate}
\item[(i)] $M(\cdot)$ is asymptotically stable;
\item[(ii)] $(M(\cdot),N(\cdot))$ is controllable.
\end{enumerate}
\end{proposition}

\begin{proposition}[Proposition 9 in \cite{Bolzern1988}]\label{prop:control}
Let  $M(\cdot)$, $N(\cdot)$ be given $n\times n$, $n\times m$ $\T$-periodic matrices respectively.
For any $\tau_0\in [0,\T)$, let $ {D}(\tau_0)$ be a matrix such that
\[
 {D}(\tau_0) {D}(\tau_0)^\tr=\int_{\tau_0}^{\tau_0+\T} \Phi_M(\tau_0+\T,s)N(s)N(s)^\tr\Phi_M(\tau_0+\T,s)^{\tr}\,\mathrm{d} s.
\]
Then the pair  $(M(\cdot),N(\cdot))$ is controllable if and only if the  time-invariant pair
$(\bar{\Phi}_M(\tau_0),  {D}(\tau_0))$ is controllable.
\end{proposition}

\bibliographystyle{siam} 

\bibliography{./gad,./ms}

\begin{thebibliography}{10}

\bibitem{Berglund2004}
{\sc N.~Berglund and B.~{Gentz}}, {\em On the noise-induced passage through an
  unstable periodic orbit {I} : Two-level model}, J. Stat. Phys, 114 (2004),
  pp.~1577--1618.

\bibitem{McC2005}
{\sc S.~Beri, R.~Mannella, D.~G. Luchinsky, A.~N. Silchenko, and P.~V.~E.
  McClintock}, {\em Solution of the boundary value problem for optimal escape
  in continuous stochastic systems and maps}, Phys. Rev. E, 72 (2005),
  p.~036131.

\bibitem{Bittanti1986}
{\sc S.~Bittanti}, {\em Deterministic and stochastic linear periodic systems},
  in Time Series and Linear Systems, S.~Bittanti, ed., Lecture Notes in Control
  and Information Sciences, Springer-Verlag, Berlin, New York, 1986,
  pp.~141--182.

\bibitem{Bolzern1988}
{\sc P.~Bolzern and P.~Colaneri}, {\em The periodic {L}yapunov equation}, SIAM
  J. MATRIX ANAL. APPL., 9 (1988), pp.~499--512.

\bibitem{BouchetPert2016}
{\sc F.~Bouchet, K.~Gawedzki, and C.~Nardini}, {\em Perturbative calculation of
  quasi-potential in non-equilibrium diffusions: a mean-filed example}, Journal
  of Statistical Physics, 163 (2015), pp.~1157--1210.

\bibitem{Bressloff2014stocell}
{\sc P.~C. Bressloff}, {\em Stochastic processes in cell biology}, vol.~41,
  Springer, 2014.

\bibitem{Cameron:2012physicaD}
{\sc {Cameron, M K}}, {\em {Finding the quasipotential for nongradient SDEs}},
  Physica D, 241 (2012), pp.~1532--1550.

\bibitem{Coddington1955}
{\sc E.~A. Coddington and N.~Levinson}, {\em Theory of Ordinary Differential
  Equations}, McGraw-Hill, 1955.

\bibitem{Day1993}
{\sc M.~V. Day}, {\em Exit cycling for the van de pol oscillattor and
  quasipotential calculations}, unpublished,  (1993).

\bibitem{PRL2018MAPOsc}
{\sc R.~de~la Cruz, R.~Perez-Carrasco, P.~Guerrero, T.~Alarcon, and K.~M.
  Page}, {\em Minimum action path theory reveals the details of stochastic
  transitions out of oscillatory states}, Phys. Rev. Lett., 120 (2018),
  p.~128102.

\bibitem{Dieci1994}
{\sc L.~Dieci and T.~Eirola}, {\em Positive definiteness in the numerical
  solution of {R}iccati differential equations}, Numer. Math., 67 (1994),
  pp.~303--313.

\bibitem{weinan-MAM2004}
{\sc W.~E, W.~Ren, and E.~Vanden-Eijnden}, {\em Minimum action method for the
  study of rare events}, Comm. Pure Appl. Math., 57 (2004), pp.~637--656.

\bibitem{FW2012}
{\sc M.~I. Freidlin and A.~D. Wentzell}, {\em Random Perturbations of Dynamical
  Systems}, Grundlehren der mathematischen Wissenschaften, Springer-Verlag, New
  York, 3~ed., 2012.

\bibitem{Heymann2006}
{\sc M.~Heymann and E.~Vanden-Eijnden}, {\em The geometric minimum action
  method: a least action principle on the space of curves}, Comm. Pure Appl.
  Math., 61 (2008), pp.~1052--1117.

\bibitem{Holland1978}
{\sc C.~J. Holland}, {\em Stochastically perturbed limit cycles}, J. Appl.
  Prob., 15 (1978), pp.~311--320.

\bibitem{KuramotoBook}
{\sc Y.~Kuramoto}, {\em Nonlinear Oscillations, Dynamical Systems, and
  Bifurcation of Vecor Fields}, Springer-Verlag, Tokyo, 1984.

\bibitem{Kurrer1991}
{\sc C.~Kurrer and K.~Schulten}, {\em Effect of noise and perturbations on
  limit cycle systems}, Physica D: Nonlinear Phenomena, 50 (1991), pp.~311 --
  320.

\bibitem{Landau-Lifshitz-Mechanics}
{\sc L.~Landau and E.~Lifshitz}, {\em Mechanics}, Course of Theoretical
  Physics, Butterworth-Heinemann, 3rd~ed., 1976.

\bibitem{Maier1996PRL}
{\sc R.~S. Maier and D.~L. Stein}, {\em Oscillatory behavior of the rate of
  escape through an unstable limit cycle}, Phys. Rev. Lett., 77 (1996),
  pp.~4860--4863.

\bibitem{MatSchuss1982}
{\sc B.~J. Matkowsky and Z.~Schuss}, {\em Diffusion across characteristic
  boundaries}, SIAM J. Appl. Math., 42 (1982), p.~822.

\bibitem{McCbook1989}
{\sc F.~Moss and P.~V.~E. McClintock}, {\em Noise in Nonlinear Dynamical
  Systems}, vol.~3, Cambridge University Press, 1989.

\bibitem{PraticalAlg}
{\sc T.~Parker and L.~Chua}, {\em Practical numerical algorithms for chaotic
  systems}, Springer-Verlag, 1989.

\bibitem{Shayman1085}
{\sc M.~A. Shayman}, {\em On the phase portrait of the matrix riccati equation
  arising from the periodic control problem}, SIAM Journal on Control and
  Optimization, 23 (1985), pp.~717--751.

\bibitem{SmelyPRE1997}
{\sc V.~N. Smelyanskiy, M.~I. Dykman, and R.~S. Maier}, {\em Topological
  features of large fluctuations to the interior of a limit cycle}, Phys. Rev.
  E, 55 (1997), pp.~2369--2391.

\bibitem{Heyman2008}
{\sc E.~Vanden-Eijnden and M.~Heymann}, {\em The geometric minimum action
  method for computing minimum energy paths}, J. Chem. Phys., 128 (2008),
  p.~061103.

\bibitem{Wan2011}
{\sc X.~Wan}, {\em An adaptive high-order minimum action method}, Journal of
  Computational Physics, 230 (2011), pp.~8669 -- 8682.

\bibitem{WantMAM2015}
{\sc X.~Wan}, {\em A minimum action method with optimal linear time scaling},
  Communications in Computational Physics, 18 (2015), pp.~1352--1379.

\bibitem{WanY_NS2dtMAM_2017}
{\sc X.~Wan and H.~Yu}, {\em A dynamic-solver-consistent minimum action method:
  {With} an application to 2d {Navier}-{Stokes} equations}, Journal of
  Computational Physics, 331 (2017), pp.~209--226.

\bibitem{WanYE_NS2dMAM_2015}
{\sc X.~Wan, H.~Yu, and W.~E}, {\em Model the nonlinear instability of
  wall-bounded shear flows as a rare event: a study on two-dimensional
  {Poiseuille} flow}, Nonlinearity, 28 (2015), p.~1409.

\bibitem{WanYZ_tMAMconv_2017}
{\sc X.~Wan, H.~Yu, and J.~Zhai}, {\em Convergence analysis of a finite element
  approximation of minimum action methods}, arXiv:1710.03471 [math], to appear
  on SIAM J. Numer. Anal. 2018,  (2017).

\bibitem{KS-WZE2009}
{\sc X.~Wan, X.~Zhou, and W.~E}, {\em Study of noise-induced transition and the
  exploration of the configuration space for the {Kuromoto-Sivachinsky}
  equation using the minimum action method}, nonlinearity, 23 (2010).

\bibitem{aMAM2008}
{\sc X.~Zhou, W.~Ren, and W.~E}, {\em Adaptive minimum action method for the
  study of rare events}, J. Chem. Phys., 128 (2008), p.~104111.

\end{thebibliography}

\end{document}